\documentclass[reqno,10pt]{amsart}

\usepackage{amssymb,amsthm,graphicx,color}
\usepackage[utf8]{inputenc}
\usepackage{enumerate}
\usepackage{enumitem}
\usepackage[margin=1in,marginpar=2cm]{geometry}
\usepackage{setspace}
\usepackage{dsfont}
\usepackage{xcolor}

\usepackage[colorlinks=true]{hyperref}
\hypersetup{urlcolor=blue, citecolor=red}

\usepackage[capitalize,nameinlink,noabbrev]{cleveref}

\numberwithin{equation}{section}

\newtheorem{theorem}{Theorem}[section]
\newtheorem{definition}{Definition}[section]
\newtheorem{prop}{Proposition}[section]
\newtheorem{remark}{Remark}[section]

\newtheorem{lemma}{Lemma}[section]

\newtheorem{cor}{Corollary}[section]

\newcommand{\RR}{\mathbb{R}}
\newcommand{\NN}{\mathbb{N}}

\newcommand{\bu}{\mathbf u}
\newcommand{\bv}{\mathbf v}
\newcommand{\bw}{\mathbf w}
\newcommand{\bx}{\mathbf x}

\newcommand{\balpha}{\boldsymbol \alpha}
\newcommand{\bF}{\mathbf F}
\newcommand{\f}{\mathbf f}
\newcommand{\rw}{\textrm{\rm w}}

\newcommand{\rc}{\textrm{\rm c}}
\newcommand{\rloc}{\textrm{\rm loc}}
\newcommand{\diss}{\textrm{\rm diss}}

\newcommand{\mM}{\mathcal M}
\newcommand{\mX}{\mathcal X}
\newcommand{\mA}{\mathcal A}
\newcommand{\mB}{\mathcal B}

\newcommand{\mC}{\mathcal C}

\newcommand{\mK}{\mathcal K}

\newcommand{\mP}{\mathcal P}

\newcommand{\mY}{\mathcal Y}
\newcommand{\mU}{\mathcal U}
\newcommand{\mV}{\mathcal V}
\newcommand{\mE}{\mathcal E}

\newcommand{\fK}{\mathfrak{K}}

\newcommand{\ve}{\varepsilon}

\newcommand{\wsconv}{\stackrel{w*}{\rightharpoonup}}
\newcommand{\wconv}{\stackrel{wsc*}{\rightharpoonup}}
\newcommand{\rd}{{\text{\rm d}}}
\newcommand{\Cloc}{\mathcal{C}_{\rloc}}
\newcommand{\MX}{\mathcal{M}(X)}

\newcommand{\MXt}{\mathcal{M}(X,\rm{tight})}
\newcommand{\MmXt}{\mathcal{M}(\mX,\rm{tight})}
\newcommand{\PXt}{\mathcal{P}(X,\rm{tight})}
\newcommand{\re}{\preccurlyeq}

\newcommand{\Wr}{W_\sigma^{1,r}}
\newcommand{\Winf}{W_\sigma^{1,\infty}}
\newcommand{\Winfp}{(W_\sigma^{1,\infty})'}
\newcommand{\Wrp}{(W_\sigma^{1,r})'}
\newcommand{\Wrw}{(W_\sigma^{1,r})_{\rm{w}}}

\newcommand{\Tsp}{\mathcal{D}_\sigma} 
\newcommand{\bunu}{\bu^\nu}
\newcommand{\bunup}{\bu^{\nu'}}
\newcommand{\bP}{\mathbb{P}}
\newcommand{\UE}{\mU^\diss}

\begin{document}

\title{On the convergence of trajectory statistical solutions}

\author{Anne C. Bronzi}
\address{Instituto de Matem\'atica, Estat\'istica e Computa\c c\~ao Cient\'ifica,
Universidade Estadual de Campinas (UNICAMP), Brazil}
\email{acbronzi@unicamp.br}

\author{Cecilia F. Mondaini}
\address{Department of Mathematics, Drexel University, USA}
\email{cf823@drexel.edu}

\author{Ricardo M. S. Rosa}
\address{Instituto de Matem\'atica, Universidade Federal do Rio de Janeiro (UFRJ), Brazil}
\email{rrosa@im.ufrj.br}

\begin{abstract} 
In this work, a recently introduced general framework for trajectory statistical solutions is considered, and the question of convergence of families of such solutions is addressed. Conditions for the convergence are given which rely on natural assumptions related to a~priori estimates for the individual solutions of typical approximating problems. The first main result is based on the assumption that the superior limit of suitable families of compact subsets of carriers of the family of trajectory statistical solutions be included in the set of solutions of the limit problem. The second main result is a version of the former in the case in which the approximating family is associated with a well-posed system. These two results are then applied to the inviscid limit of incompressible Navier-Stokes system in two and three spatial dimensions, showing, in particular, the existence of trajectory statistical solutions to the two- and three-dimensional Euler equations, in the context of weak and dissipative solutions, respectively. Another application of the second main result is on the Galerkin approximations of statistical solutions of the three-dimensional Navier-Stokes equations.
\end{abstract}



\maketitle

\tableofcontents
\newpage

\section{Introduction}\label{sec:Intro}

The theory of statistical solutions for the Navier-Stokes equations, initiated by Foias and Prodi in the early 1970s \cite{Foias72, Foias73, FoiasProdi1976}, and followed by Vishik and Fursikov later in the decade \cite{VishikFursikov77, VishikFursikov78}, has seen many advances in relation to the theory of turbulence in fluid flow problems \cite{Foias74, LadyVershik1977, CFM94, BCFM1995, DoeringTiti1995, Foias97, Fursikov99, FMRT2001, FMRT2001c, Basson2006, ConstantinRamos2007, RRT08, Wang2008, Kelliher2009, Rosa2009, BFMT2019}. The concept has also been adapted to a number of other specific equations, or specific classes of equations, for which a well-defined evolution semigroup does not exist or is not known to exist \cite{IllnerWick1981, Golec1993, Barbablah1995, Haragus1999, Slawik2008, BMR2014}. Later on, \cite{FRT2013} developed a slightly different formulation of the statistical solutions in \cite{VishikFursikov77, VishikFursikov78} which is compatible with those in \cite{Foias72, Foias73, FoiasProdi1976}, in the way that projections of the solutions in this new formulation become solutions in the sense of the latter.

More recently, inspired by \cite{FRT2013}, the authors introduced, in \cite{BMR2016}, a general framework applicable to a wide range of equations and containing the main results on the existence of statistical solutions for an associated initial-valued problem, based on natural and amenable conditions. This greatly facilitates the application of the notion of statistical solution both to the Navier-Stokes equations and to other systems. In this manuscript, we address the problem of \emph{convergence} of families of statistical solutions, within this general framework.

Approximating a given problem is of fundamental importance in every branch of Mathematics, both pure and applied, and in many scientific fields and applications. This is no different in regards to statistical solutions. Be it with the aim of proving the existence of solutions, or in relation to asymptotic analysis, numerical computations and other investigations on the nature of the given problem. With this in mind, we present a couple of general results on the limit of trajectory statistical solutions and apply the results to the Euler equations as inviscid limits of the Navier-Stokes equations and to the Galerkin approximations of the Navier-Stokes equations. The inviscid limit is treated both in the two- and three-dimensional cases. In particular, our results show the existence of trajectory statistical solution to the limit Euler equations. The Galerkin approximation is considered in the three-dimensional viscous case, yielding, in particular, a different proof of existence of statistical solutions for the three-dimensional Navier-Stokes equations than given in \cite{FRT2013,BMR2016}, see \cref{rmk:exist:ss:3D:NSE}.

Our two main results concern conditions for the limit of trajectory statistical solutions, of a given family of problems, to be a trajectory statistical solution of the limit problem. The first result, \cref{Main1}, concerns the more general case of a family of trajectory statistical solutions associated with an approximate problem which is not necessarily well-posed. The second result, \cref{Main2alt}, concerns the special case in which the family of trajectory statistical solutions of the approximate problem is associated with a well-defined semigroup of individual solutions, although, at the limit, the well-posedness may not stand.

We recall that a $\mathcal{U}$-trajectory statistical solution, as introduced in \cite{BMR2016}, is a Borel probability measure $\rho$ which is tight and is carried by a Borel subset of a set $\mathcal{U}$ on a space $\mathcal X=\mathcal C_{loc}(I,X)$ of continuous functions from a real interval $I$ into a Hausdorff space $X$, with $\mathcal{X}$ endowed with the topology of uniform convergence on compact subsets of $I$ (see \cref{def-stat-sol}). The term \emph{solution} here is in regards to the \emph{problem} of finding such Borel probability measure carried inside $\mathcal{U}$. This is not a problem \emph{per se,} because any Dirac measure carried at an individual element in $\mathcal{U}$ is a trajectory statistical solution, but the problem becomes nontrivial when associated with an initial condition $\Pi_{t_0}\rho = \mu_0$, where $\Pi_{t_0}:\mathcal{X} \rightarrow X$ is the projection at a ``initial" time $t_0\in I$ and $\mu_0$ is a given ``initial" measure on $X$. In applications, $\mathcal{U}$ is, for example, associated with the set of solutions of a given differential equation, with values in a phase space $X$, such as the Leray-Hopf weak solutions of the 3D Navier-Stokes equations.

Suppose now that we have a family $\{\rho_\ve\}_{\ve \in \mE}$ of $\mU_\ve$-trajectory statistical solutions, for some index set $\mE$ and with $\mU_\ve \subset \mX$. Think of $\mU_\ve$ as the space of solutions of a given differential equation depending on a parameter $\ve$ (e.g. the Leray-Hopf weak solutions of the 3D Navier-Stokes equations with the viscosity as the parameter, or a numerical approximation of a given equation depending on a certain discretization parameter) and of $\mU$ as a set of solutions of the limit problem (e.g. dissipative solutions of the Euler equations). We establish, in \cref{Main1}, suitable conditions on the net $\{\rho_\ve\}_{\ve \in \mE}$ so that it converges (passing to a subnet if necessary), in a certain sense, to a $\mU$-trajectory statistical solution $\rho$. More precisely, \cref{Main1} imposes a uniform tightness assumption on $\{\rho_\ve\}_\ve$ along with a suitable condition relating the limit of compact subsets of $\mU_\ve$ to $\mU$, i.e. concerning the limit of approximate individual solutions.

\cref{Main2alt} treats the special case in applications where each member in the family of approximating equations possesses a well-defined solution operator $S_\ve: X \to \mX$. Namely, $S_\ve$ takes each initial datum $u_{0,\ve} \in X$ to the unique trajectory $u_\ve = u_\ve(t)$ in $\mX$ of the corresponding approximate equation satisfying the initial condition $u_\ve(t_0) = u_{0,\ve}$. In this case, a set of assumptions is required from the family of operators $\{S_\ve\}_{\ve \in \mE}$ to guarantee that, given any initial probability measure $\mu_0$ on $X$ and suitable measures $\{\mu_{\ve}\}_{\ve \in \mE}$ on $X$ approximating $\mu_0$ in a certain sense, the family of measures $\rho_\ve = S_\ve \mu_\ve$, $\ve \in \mE$, on $\mX$ has a convergent subnet to a $\mU$-trajectory statistical solution $\rho$ such that $\Pi_{t_0} \rho = \mu_0$. 

In applications, such assumptions translate to verifying the following for the approximating equations: (i) continuous dependence of solutions of each approximating equation with respect to initial data lying in compact sets; (ii) suitable parameter-uniform a priori estimates; and (iii) convergence of individual solutions of the approximating systems starting from a fixed compact set towards a solution of the limit equation. This latter condition is weaker than the corresponding condition in \cref{Main1} and more natural in this context, and for this reason we do not apply \cref{Main1} directly to prove \cref{Main2alt}; see \cref{rmkthm2vsthm1}.

The applications are treated in detail in \cref{secapplications}. The first example, in \cref{subsec:2D:NSE:Euler}, concerns the inviscid limit of the Navier-Stokes equations to the Euler equations in the two-dimensional case, illustrating an application of \cref{Main2alt}. We consider, more precisely, the set of weak solutions of the 2D Euler equations with periodic boundary conditions and initial vorticity in $L^r$, with $2 \leq r < \infty$. In this case, a semigroup is not known to exist at the inviscid limit, but the viscous approximation has a well defined semigroup associated to weak solutions of the 2D Navier-Stokes equations under this setting (see \cref{MainThmApp2D}).

The second example, in \cref{subsec:3D:NSE:Euler}, is the inviscid limit of the Navier-Stokes equations to the Euler equations in three dimensions, where we consider periodic boundary conditions and initial data in $L^2$. For the 3D Euler equations, we consider the corresponding set of \emph{dissipative solutions}, whereas for the 3D Navier-Stokes equations we consider the set of \emph{Leray-Hopf weak solutions}. This provides an application of \cref{Main1}, given that no semigroup is available under this setting, neither at the limit, nor for the approximating family (see \cref{MainThmApp3D}). Here, the verification of item (iii), namely the convergence of individual solutions of the 3D Navier-Stokes equations towards a solution of 3D Euler is in general a delicate issue, but is nevertheless known to hold within the context of dissipative solutions for Euler. See a more detailed discussion on this in \cref{subsec:setting}, and also \cref{thm3dNStoEuler}.

Our final example is given in \cref{subsec:3D:Galerkin:NSE} and concerns a spectral spatial discretization of the 3D Navier-Stokes equations given by the standard Galerkin approximation. Since the initial-value problem associated to the Galerkin system is well-posed, this gives another application of \cref{Main2alt}. 
Moreover, as we point out in \cref{rmk:Pm:mu0} and \cref{rmk:mu0m:montecarlo} below, the flexibility provided by the framework of \cref{Main2alt} with the approximating initial measures $\{\mu_\ve\}_{\ve \in \mE}$ allows not only for the natural example given by Galerkin projections of a given initial measure $\mu_0$ for the limiting system, namely $\mu_\ve = \mu_m = P_m \mu_0$, but also the example of Monte Carlo approximations of $\mu_0$, provided the required conditions from \cref{Main2alt} are met. 

Regarding the first two applications, we first note that the two-dimensional inviscid limit that we present here for the periodic case has also been treated by Wagner and Wiedemann \cite{WW2022}, in the whole space, with no forcing term. The result in \cite{WW2022} uses the general framework given in \cite{BMR2016}, and we take this opportunity to show how the general framework developed here, specifically in \cref{Main2alt}, can be applied to simplify the corresponding proof. Moreover, differently than \cite{WW2022}, our application allows the presence of a time-dependent forcing term.

In regard to the application for the inviscid limit in 3D, we point out that a similar question was tackled in \cite{FjordholmMishraWeber2024}, although in the context of a certain notion of statistical solution in phase space called \emph{Friedman-Keller statistical solution} \cite{FjordholmLanthalerMishra2017,FjordholmLyeMishraWeber2020}, which particularly takes into account the temporal evolution of multi-point spatial correlations in the flow. Specifically, \cite{FjordholmMishraWeber2024} shows two main results. The first concerns an equivalence between the definitions of Friedman-Keller statistical solution and Foias-Prodi statistical solution from \cite{Foias72, Foias73, FoiasProdi1976} for the 3D Navier-Stokes equations. Secondly, it is shown that, under a certain statistical scaling assumption, any suitable sequence of Friedman-Keller statistical solutions for the 3D Navier-Stokes equations converges, up to a subsequence, towards a corresponding one for the 3D Euler equations under the inviscid limit, in a certain time-averaged sense.

While the statements of our general convergence results, \cref{Main1} and \cref{Main2alt}, are given with respect to statistical solutions in trajectory space, we note that a general notion of statistical solution in phase space was given in our previous work, \cite{BMR2016}. These consist of time-parametrized families of Borel probability measures on $X$ satisfying suitable regularity conditions and a Liouville-type equation associated with an evolution equation $\mathrm{d}u/\mathrm{d}t = F(t, u)$. We also show in \cite{BMR2016} that any suitably regular trajectory statistical solution $\rho$ on $\mX$ yields a phase-space  statistical solution via the time projections $\mu_t = \Pi_t\rho$, $t \in I$, called a \emph{projected statistical solution}, although the converse might not necessarily hold. As such, if $\{\Pi_t \rho_\ve \}_{t \in I}$, $\ve \in \mE$, is a collection of projected statistical solutions for which $\{\rho_\ve\}_{\ve \in \mE}$ satisfies our conditions guaranteeing convergence to a trajectory statistical solution $\rho$ in $\ve$, then it follows immediately from the continuity of the projection operator $\Pi_t$, together with the notion of convergence in the space of probability measures we consider, that $\Pi_t \rho_\ve$ converges to $\Pi_t \rho$ in $\ve$, for each $t \in I$. Therefore, under such conditions and provided $\rho$ is sufficiently regular, the projected statistical solutions $\{\Pi_t \rho_\ve \}_{t \in I}$, $\ve \in \mE$, converge to a projected statistical solution of the limit problem, namely $\{\Pi_t \rho\}_{t \in I}$.

Moreover, we recall that, as shown in one of the applications from \cite{BMR2016}, any suitable projected statistical solution of the Navier-Stokes equations is also a Foias-Prodi statistical solution, and hence it corresponds to a Friedman-Keller statistical solution as in \cite{FjordholmMishraWeber2024}. Thus, in comparison to the aforementioned inviscid convergence result from \cite{FjordholmMishraWeber2024}, it would be interesting to investigate: (a) whether a projected statistical solution of the Euler equations corresponds to a solution in the Friedman-Keller sense, or, more generally, whether the notion of phase-space statistical solutions of the Euler equations, in the sense of \cite{BMR2016}, is equivalent to the concept of Friedman-Keller statistical solutions for such equations; 
and (b) what is the relation between the required assumptions and type of convergence from the result in \cite{FjordholmMishraWeber2024} and the  pointwise-in-time convergence for projected statistical solutions implied by our current results. We leave such investigation for future work.

We emphasize that, besides these illustrative applications, the framework is quite general and several other limiting problems fit within the scope of our theory. For example, the results apply to various different numerical discretizations; to the viscous approximations of the inviscid magnetohydrodynamic (MHD) equations; to the NSE-$\alpha$ and MHD-$\alpha$ models as regularized approximations of the NSE and MHD equations, respectively; to compressible approximations of the incompressible 3D NSE equations; and to approximations of other models such as reaction-diffusion equations and nonlinear wave equations \cite{BronziRosa2014,BMR2016}.

To further connect with the existing literature, we mention that there are a number of previous results concerning the limit of statistical solutions in various senses, with the majority given in the context of the Navier-Stokes equations. Those include the inviscid limit of the Navier-Stokes equations (or a damped version of it) towards the Euler equations in \cite{Chae1991, ConstantinRamos2007}, besides the previously mentioned work \cite{WW2022}; the vanishing $\alpha$ limit of the $\alpha$-Navier-Stokes equations \cite{VTC2007, BronziRosa2014}; the Navier-Stokes equations as the limit of the viscoelastic Navier-Stokes-Voigt model \cite{RamosTiti2010}; and the infinite Prandtl number limit of Rayleigh-B\'enard convection \cite{Wang2008}. Our results, however, instead of focusing on a specific model, apply to a general framework that can be more easily verified in specific cases.

It is also worth mentioning a yet another notion of statistical solution based on the concept of Markov selection, developed in \cite{CardonaKapitanskii2020}, for general class of evolutionary problems, and in \cite{FanelliFeireisl2020}, in the context of weak solutions of the barotropic Navier-Stokes systems, inspired by previous works in the context of stochastic equations \cite{FlandoliRomito2008, BreitFeireislHofmanova2020}. The connection of this type of statistical solution with our theory and the limiting process of such solutions is not clear but is currently under investigation.

This manuscript is organized as follows. \cref{sec:prelim} recalls the necessary background on certain functional analytical and measure theoretical tools, including the notions of convergence in spaces of probability measures that we consider. Our general results on convergence of statistical solutions are given in \cref{secAbstractRes}. Finally, \cref{secapplications} presents our applications of these general results, namely the convergence of trajectory statistical solutions in the 2D and 3D inviscid limits, and the Galerkin approximations.

\section{Preliminaries}\label{sec:prelim}

In this section, we briefly recall the basic topological and measure theoretical concepts underlying our framework. For further details, we refer to e.g. \cite{AB,Bogachev2007,BMR2016,Folland,Topsoebook}.

\subsection{Functional setting}

Given a topological vector space  $X$, we denote its dual by $X'$ and the corresponding duality product as $\langle \cdot, \cdot\rangle_{X',X}$.
We employ the notation $X_\rw$ to indicate that $X$ is endowed with its weak topology, whereas $X_{\rw^*}'$ stands for $X'$ endowed with the weak-star topology. Notice that, for any topological vector space $X$, the space $X_{\rw^*}'$ is always a Hausdorff locally convex topological vector space (\cite[Section 1.11.1]{Edwards1965}). Further, if $X$ is in particular a Banach space, we denote its norm by 
$\|\cdot\|_X$, and by $\|\cdot\|_{X'}$ the usual operator norm in the dual space $X'$.

For any Hausdorff space $X$ and interval $I \subset \RR$, we denote by $\mC(I,X)$ the \textbf{space of continuous paths}
in $X$ defined on $I$, i.e. the space of all functions $u: I \rightarrow X$ which are continuous. When $\mC(I,X)$ is endowed with the compact-open topology, we denote $\mX = \Cloc(I,X)$. Here, the subscript ``\emph{\rloc}'' refers to the fact that this topology is based on compact subintervals of $I$. We recall that when $X$ is a uniform space, the \textbf{compact-open topology} in $\Cloc(I,X)$ coincides with the \textbf{topology of uniform convergence on compact subsets} 
\cite[Theorem 7.11]{Kelley75}. In particular, this holds when $X$ is a topological vector space, which is the case in both applications presented in \cref{secapplications}.

For any $t\in I$, let $\Pi_{t}:\mX\rightarrow X$ be the ``projection'' map at time $t$, defined by
\begin{align}\label{defprojectionop}
\Pi_{t}u = u(t), \quad \mbox{ for all } u \in \mX.
\end{align}
Moreover, given any subset $I' \subset I$, define $\Pi_{I'}: \Cloc(I,X) \to \Cloc(I',X)$ to be the restriction operator 
\begin{align}\label{def:restriction:op}
	\Pi_{I'} u = u|_{I'}, \quad \mbox{ for all } u \in \Cloc(I,X).
\end{align}
It is readily verified that $\Pi_t$ and $\Pi_{I'}$ are continuous with respect to the compact-open topology.

We also consider the space of bounded and continuous real-valued functions on $X$, denoted by $\mC_b(X)$. When $X$ is a subset of 
$\mathbb R^m$, $m\in \mathbb N$, we further consider the space $\mC_c^\infty(X)$ of infinitely differentiable real-valued functions on $X$ which 
are compactly supported in the interior of $X$.

When $X$ is a Hausdorff topological space, $(\Gamma,\preccurlyeq)$ is a directed set, and $\{A_\gamma\}_{\gamma\in \Gamma}$ is a net in $X$, we recall the definitions of \textbf{topological inferior and superior limits} (see \cite[Definition 2.2.3]{BIbook}, where they are called \emph{interior and exterior limits,} and also \cite[Definition 5.2.1 and Proposition 5.2.2]{B93}, where they are called \emph{lower and upper closed limits}, respectively):
\[\liminf_\gamma A_{\gamma}=\bigcap\left\{ \overline{\bigcup_{\lambda \in \Lambda} A_\lambda}: \Lambda \mbox{ is a cofinal subset of }\Gamma\right\},\]
\[\limsup_\gamma A_{\gamma}=\bigcap\left\{ \overline{\bigcup_{\lambda \in \Lambda} A_\lambda}: \Lambda \mbox{ is a terminal subset of }\Gamma\right\}.\]
In these definitions, the overline denotes the closure of the set under the topology of $X$; a subset $\Lambda$ is cofinal in $\Gamma$ when for every $\gamma\in \Gamma$, there exists $\lambda \in \Lambda$ such that $\gamma \preccurlyeq \lambda$ (like a subsequence); and a subset $\Lambda$ is terminal in $\Gamma$ when there exists $\gamma\in \Gamma$ such that $\Lambda = \{\lambda\in\Gamma; \;\gamma \preccurlyeq \lambda\}$ (like the tails of a sequence). Note that these definitions of topological inferior and superior limits are different from the set-theoretic limits, where no topological closure is used in the definitions.

When both limits agree, we define the result as the \textbf{topological limit} of the net: 
\[ \lim_\gamma A_\gamma = \limsup_\gamma A_\gamma = \liminf_\gamma A_\gamma.
\]

In this work, we only use the superior limit, for which the following characterization is useful, relating it to the classical definition for sequences (see \cite[Theorem 2.10]{F60}):
\begin{align}\label{def:limsup}
	\limsup_\gamma A_\gamma = \bigcap_{\gamma\in \Gamma}\overline{\bigcup_{\gamma \preccurlyeq \beta} A_\beta}.
\end{align}

\subsection{Elements of measure theory}\label{subsec:meas:theory}

Consider a topological space $X$ and let $\mB_X$ denote the $\sigma$-algebra of Borel sets in $X$. We denote by $\mM(X)$ the set of finite Borel measures in $X$, and by $\mP(X)$ its subset of all Borel probability measures in $X$.

Given a family $\fK$ of Borel sets in $X$, we say that a measure $\mu\in \mM(X)$ is \textbf{inner regular with respect to} $\fK$ if for every set $A\in \mB_X$,
\[
\mu(A) = \sup \{\mu(K) : K\in \fK, \;K \subset A\}.
\]
A measure $\mu\in \mM(X)$ is \textbf{tight} or \textbf{Radon} if $\mu$ is inner regular with respect to the family of all compact subsets of $X$. Moreover, a measure $\mu\in \mM(X)$ is called \textbf{outer regular} if for every set $A\in \mB_X$,
\[
\mu(A) = \inf \{\mu(U) : U \mbox{ is open}, \;A \subset U\}.
\]

A net $\{\mu_\gamma\}_{\gamma \in \Gamma}$ of measures in $\mM(X)$ is said to be \textbf{uniformly tight} if for every $\varepsilon>0$ there
exists a compact set $K\subset X$ such that 
\[
\mu_\gamma(X\setminus K)<\varepsilon, \quad \mbox{ for all } \gamma \in \Gamma.
\]
Here we follow the terminology in \cite{Bogachev2007}, and we remark that such concept of uniform tightness does not necessarily imply tightness of each measure $\mu_\gamma$. In what follows, however, we consider uniformly tight families of tight measures.

We denote the set of all measures $\mu \in \MX$ which are tight by $\MXt$, and its subset encompassing all tight Borel probability measures by $\PXt$.

Let us recall some useful facts regarding these definitions. First, every tight finite Borel measure on a Hausdorff space $X$ is outer regular, see \cite[Theorem 12.4]{AB}. Furthermore, if $X$ is a Polish space then every finite Borel measure on $X$ is tight, see \cite[Theorem 12.7]{AB}. We will make use of this latter result in \cref{secapplications} in connection with the fact that in every separable Banach space $X$, the Borel $\sigma$-algebras generated by the strong and weak topologies coincide, i.e. $\mB_X = \mB_{X_\rw}$, see e.g. \cite[Section 2.2]{Mondainithesis}.

Now let $X$ and $Y$ be Hausdorff spaces and consider a Borel measurable function $F: X\rightarrow Y$. Then every measure $\mu$ on $\mB_X$ induces a measure $F\mu$ on $\mB_Y$ known as the \textbf{push-forward of $\mu$ by $F$} and defined as 
\[
F\mu(E) = \mu(F^{-1}(E)), \quad \mbox{for all } E\in\mB_Y.
\]
Moreover, if $\varphi: Y\rightarrow \mathbb R$ is a $F\mu$-integrable function then $\varphi\circ F$ is
$\mu$-integrable and the following change of variables formula holds
\begin{equation}\label{induced-measure}
\int_X \varphi(F(x))d\mu(x)=\int_Y \varphi(y)dF\mu(y),
\end{equation}
see e.g. \cite{AB}. Clearly, if $\mu$ is a tight measure and $F$ is a continuous function then the push-forward measure $F\mu$ is also tight.

Next, we present a generalization for nets of the continuity of a
finite measure with respect to a decreasing sequence of measurable sets (see
\cite[Theorem 1.8]{Folland} or \cite[Lemma 4.51]{AB}). The following proof is based on similar ideas from \cite[Proposition 10]{kelley2}.

\begin{lemma}\label{contabove}(Continuity from above)
Let $X$ be a compact Hausdorff space and let $\mu\in \mM(X)$ be an outer regular
measure. Then, for any  monotone decreasing
net $(E_\gamma)_{\gamma\in \Gamma}$ of compact sets in $X$,
$E=\bigcap_{\gamma\in\Gamma}E_\gamma$ is a compact set and
\[\mu(E)=\lim_{\gamma\in\Gamma} \mu(E_\gamma).\]
\end{lemma}
\begin{proof}
It is clear that $E=\bigcap_{\gamma\in\Gamma}E_\gamma$ is a closed subset of a
compact set so that it is compact. Since $\mu$ is outer regular
then, for every $\ve>0$, there exists an open set $U\subset X$ such that
$E\subset U$ and $\mu(U)<\mu(E)+\ve$. Observe that the compactness of $X$
implies the compactness of $U^c$ and, since each $E_\gamma$ is compact, we also
have that $E_\gamma^c$ is an open set on $X$. Furthermore,
$U^c\subset \bigcup_{\gamma\in\Gamma}E_\gamma^c$ so that there exists a
finite subset $\{\gamma_1,\ldots,\gamma_N\}\subset \Gamma$ such that
$\bigcap_{i=1}^N E_{\gamma_i}\subset U$. Since $(\Gamma,\preccurlyeq)$ is a directed set
then there exists $\bar{\gamma}\in\Gamma$ such that $\gamma_i\preccurlyeq\bar{\gamma}$,
for all $i=1,\ldots,N$. Since the net $(E_\gamma)_{\gamma\in \Gamma}$ is monotone decreasing, we have that $E_{\bar\gamma}\subset E_{\gamma_i}$ for every $i=1, \ldots, N$.  Hence, $E_{\bar{\gamma}}\subset \bigcap_{i=1}^N E_{\gamma_i}\subset U$. Therefore,
$\inf_{\gamma\in \Gamma}\mu(E_\gamma)\leq \mu(U) <\mu(E)+\ve$. Since $\ve>0$ is
arbitrary we conclude that $\inf_{\gamma\in\Gamma}\mu(E_\gamma)\leq\mu(E)$. On
the other hand, it is clear that $\mu(E)\leq
\inf_{\gamma\in\Gamma}\mu(E_\gamma)$. Thus,
\[\lim_{\gamma\in\Gamma}
\mu(E_\gamma)=\inf_{\gamma\in\Gamma}\mu(E_\gamma)=\mu(E),\]
as desired.
\end{proof}

\subsection{Topologies for measure spaces and related results}\label{subsec:top:meas}

We recall the definitions of two specific topologies in $\MX$, for any topological space $X$. First, the \textbf{weak-star topology} is the smallest one for which the mappings $\mu \mapsto \mu(f)  = \int_X f(x) d \mu(x)$ are continuous, for every bounded and continuous real-valued function $f$ on $X$, i.e. $f \in \mC_b(X)$. If a net $\{\mu_\alpha\}_\alpha$ converges to $\mu$ with respect to this topology, we denote $\mu_\alpha \wsconv \mu$. A less common topology is the one defined by Topsoe in \cite{Topsoebook}, which is the smallest one for which the mappings $\mu \mapsto \mu(f)$ are upper semicontinuous, for every bounded and upper semi-continuous real-valued function $f$ on $X$. Topsoe calls this topology the ``weak topology'', but in order to avoid any confusion we call it here the \textbf{weak-star semicontinuity topology} on $\MX$. We denote convergence of a net $\{\mu_\alpha\}_\alpha$ to $\mu$ with respect to this latter topology as $\mu_\alpha \wconv \mu$.

From these definitions, it is not difficult to see that the weak-star topology is in general weaker than the weak-star semi-continuity topology. Moreover, according to \cref{portmanteau} below, if $X$ is a completely regular Hausdorff space, then these two topologies coincide when restricted to the space $\MXt$.

The following lemma summarizes some properties and useful characterizations for these topologies, see \cite[Theorem 8.1]{Topsoebook}.

\begin{lemma}\label{portmanteau}Let $X$ be a Hausdorff space.
For a net $\{\mu_\alpha\}_\alpha$ in $\MX$ and $\mu\in \MX$,
consider the following statements:
\begin{enumerate}[label={(\roman*)}]
\item \label{PMwscstar} $\mu_\alpha \wconv\mu$;
\item \label{PMlimsup} $\limsup \mu_\alpha (f)\leq \mu(f)$, for all bounded upper
semicontinuous function $f$;
\item \label{PMliminf} $\liminf \mu_\alpha (f)\geq \mu(f)$, for all bounded lower
semicontinuous function $f$;
\item \label{PMclosed} $\lim_\alpha \mu_\alpha(X)=\mu(X)$ and $\limsup \mu_\alpha (F)\leq \mu(F)$, for all closed set $F\subset X$;
\item \label{PMlimopen} $\lim_\alpha \mu_\alpha(X)=\mu(X)$ and $\liminf \mu_\alpha (G)\geq \mu(G)$, for all open set $G\subset X$;
\item \label{PMwstar} $\mu_\alpha \wsconv \mu$.
\end{enumerate}

Then the first five statements are equivalent and each of them implies the last one.

Furthermore, if $X$ is also completely regular and $\mu \in \MXt$, then all six statements are equivalent.
\end{lemma}

We next state a result of compactness on the space of tight measures $\MXt$ that is essential for our main result. 
For a proof of this fact, see \cite[Theorem 9.1]{Topsoebook}.

\begin{theorem}\label{topsoe}
Let $X$ be a Hausdorff space and let $\{\mu_{\alpha}\}_{\alpha}$ be a net in
$\MXt$ such that $\limsup \mu_{\alpha} (X) < \infty$. If $\{\mu_{\alpha}\}_{\alpha}$ 
is uniformly tight, then it is compact with respect to the weak-star semi-continuity topology in $\MXt$.
\end{theorem}

An important property of the space $\MXt$ is that it is Hausdorff
when endowed with the  weak-star semi-continuity (respectively, weak-star) topology whenever $X$ is a Hausdorff (respectively, completely regular Hausdorff)
space. This result was proved by Topsoe \cite{Topsoebook} but we also refer the reader to \cite[Section 2.4]{BMR2016} for a more
detailed proof.
As a consequence of this Hausdorff property and the definition of the weak-star topology, one immediately obtains the following characterization of the equality of two measures in $\MXt$ for any completely regular space $X$.

\begin{prop}\label{prop-measure}
Let $X$ be a completely regular Hausdorff space and
$\mu_1, \mu_2\in \MXt$. Then 
\begin{equation}\label{uniq-measu}
\mu_1=\mu_2 \,\,\mbox{ if and only if } \,\, \int_X\varphi(x)\rd
\mu_1(x)=\int_X\varphi(x)\rd \mu_2(x) \,\, \mbox{ for all } \varphi\in \mC_b(X). 
\end{equation}
\end{prop}

\section{Convergence of trajectory statistical solutions}\label{secAbstractRes}

This section contains our main results regarding convergence of trajectory statistical solutions in the sense defined in \cite{BMR2016}, which we recall in \cref{def-stat-sol} below. Throughout the section, we let $X$ be a Hausdorff space and denote as before by $\mathcal X=\mathcal C_{loc}(I,X)$ the space of continuous functions from $I$ into $X$ endowed with the topology of uniform convergence on compact sets. We also denote by $\mB_X$ and $\mB_\mX$ the Borel $\sigma$-algebras of $X$ and $\mX$, respectively.

As pointed out in  \cite{BMR2016}, we note that the terminology ``trajectory statistical solutions'' refers to the fact that these are measures on $\mX$ carried by a measurable subset of a fixed set $\mU \subset \mX$ which, in applications, would consist of the set of trajectories, i.e. solutions in a certain sense, of a given evolution equation. As such, these trajectory statistical solutions represent the probability distribution of all possible individual trajectories of the equation. At this abstract level, however, we do not specify the evolution equation, fixing only its corresponding set of solutions $\mU$.

 \begin{definition}\label{def-stat-sol}
Let $\mU$ be a subset of $\mX$. We say that a Borel probability measure $\rho$ on $\mX$ is a \textbf{$\mU$-trajectory statistical solution} if
\begin{enumerate}[label={(\roman*)}]
\item $\rho$ is tight;
\medskip
\item $\rho$ is carried by a Borel subset of $\mX$ included in $\mU$, i.e. there exists $\mV\in\mB_\mX$ such that $\mV\subset \mU$ and $\rho(\mX \setminus \mV)=0$.
\end{enumerate}
\end{definition}

From now on, we fix the following convention regarding notation. We use calligraphic capital letters to denote subsets of $\mX$ (e.g. $\mK$, $\mU$, $\mA$, etc.), and plain capital letters for subsets of $X$ (e.g. $K$, $A$, etc.). The letters $\mE$ or $\Gamma$ are used as index sets of nets, where the indices are usually represented by the letters $\alpha$, $\beta$, $\varepsilon$, or $\gamma$.

We prove below two theorems on the convergence of trajectory statistical solutions, as described in more details in the Introduction. One is for arbitrary trajectory statistical solutions, suitable to approximations which are not necessarily well-posed, and the other is for approximations which have a well-defined solution semigroup.

The first result stems from the compactness of sets of uniformly tight Borel probability measures. The main point is to localize the carrier of the limit trajectory statistical solution. The second one simplifies the necessary conditions in the case there is a well-defined solution semigroup.

\begin{theorem}\label{Main1}
    Let $X$ be a Hausdorff space, $I$ be an interval in $\mathbb R$, and $\{\rho_\ve\}_{\ve\in \mE}$ be a family of $\mU_\ve$-trajectory statistical solutions on subsets $\mU_\ve \subset \mX$, carried by Borel subsets $\mV_\ve \subset \mX$, where $\mX=\Cloc(I,X)$. Let $\mU \subset \mX$ and suppose there is a sequence $\{\mK_n\}_{n\in \mathbb{N}}$ of compact sets in $\mX$ such that, for all $n \in \NN$,
 \begin{enumerate}[label=(A\arabic*)]
 	\item\label{i:Main1} $\rho_\ve(\mX \setminus \mK_n) < \delta_n$, for all $\ve\in\mE$, with $\delta_n \rightarrow 0$; 
 	\item\label{ii:Main1} $\limsup_{\ve \in \mE}(\mV_{\ve}\cap \mK_n)\subset \mathcal U$ with respect to the topological superior limit in $\mX$. 
 \end{enumerate}
 Then, there exists a $\mU$-trajectory statistical solution $\rho$ which is a weak-star semicontinuity limit of a subnet of $\{\rho_\ve\}_{\ve\in \mE}$. Moreover, if the interval $I$ is closed and bounded on the left with left endpoint $t_0$ and $\Pi_{t_0} \rho_\ve \wconv \mu_0$ for some tight Borel probability measure $\mu_0$ on $X$, then $\Pi_{t_0} \rho = \mu_0$.
\end{theorem}

\begin{proof}
From the assumption that $\rho_\ve(\mX \setminus \mK_n) < \delta_n$, for all $\ve\in\mE$, with compact sets $\mK_n \subset \mX$ and $\delta_n \rightarrow 0$, it follows immediately that $\{\rho_\ve\}_{\ve\in \mE}$ is a uniformly tight family of probability measures in $\mM(\mX,\rm{tight})$. Therefore, from the compactness result in \cref{topsoe}, there exists a subnet of $\{\rho_\ve\}_{\ve\in \mE}$ which converges in the weak-star semicontinuity topology to a probability measure $\rho \in \MmXt$.
To deduce that $\rho$ is a $\mU$-trajectory statistical solution, it remains to show that $\rho$ is carried by a Borel subset of $\mU$, which is the main component of this proof.
 
Using the carriers $\mV_\ve \subset \mU_\ve$ of each $\mU_\ve$-trajectory statistical solution, define the set
 \begin{equation}\label{Vcarrierofrho}
   \mV = \bigcup_{n=1}^\infty\limsup_{\varepsilon \in \mE} (\mathcal{V}_\varepsilon\cap \mK_n)
   =\bigcup_{n=1}^\infty\bigcap_{\ve\in \mE} \overline{\bigcup_{\ve \re \gamma}  \mV_{\gamma}\cap \mK_n},
 \end{equation}
 cf. \eqref{def:limsup}. Thanks to condition \ref{ii:Main1}, it follows that $\mV\subset \mU$. Moreover, since an arbitrary intersection of closed sets is closed and since each set $\overline{\bigcup_{\ve \re \gamma}  \mV_{\gamma}\cap \mK_n}$, $\ve \in \mE$, is a closed subset of the compact subset $\mK_n$, the sets $\bigcap_{\ve \in \mE}\overline{\bigcup_{\ve \re \gamma}  \mV_{\gamma}\cap \mK_n}$, $n\in \mathbb{N}$, are compact, and $\mV$ is a $\sigma$-compact subset of $\mX$, hence Borel.
 
 We now claim that $\rho$ is carried by $\mV$. Since $\mE$ is a directed set, then given any $\beta \in \mE$ there exists $\alpha \in \mE$ such that $\beta\re\alpha$ and $\ve\re \alpha$. 
 For any such $\alpha$, we have, by assumption,
\[
 	\rho_{\alpha}\left(\overline{\bigcup_{\ve \re \gamma}\mV_{\gamma}\cap \mK_n}\right)\geq
	 \rho_{\alpha}\left(\mV_{\alpha}\cap \mK_n\right)=\rho_\alpha(\mK_n)\geq 1-\delta_n,
\]
for any given $n\in\mathbb{N}.$ Taking the supremum in $\alpha$, we find that
\[
	\sup_{\beta \re \alpha} \rho_{\alpha}\left(\overline{\bigcup_{\ve \re
			\gamma}\mV_{\gamma}\cap \mK_n}\right)
	\geq 
	\sup_{\beta \re \alpha, \, \ve \re \alpha} \rho_{\alpha}\left(\overline{\bigcup_{\ve \re
			\gamma}\mV_{\gamma}\cap \mK_n}\right)\geq 1-\delta_n.
\]
Since $\beta \in\mE$ above is arbitrary, we take the infimum of this expression over $\beta \in \mE$ and find that
\[
\limsup_{\beta \in \mE}\rho_{\beta }\left(\overline{\bigcup_{\ve \re \gamma}\mV_{\gamma}\cap \mK_n}\right)\geq 1-\delta_n.
\]

In view of \cref{portmanteau}, statement \ref{PMclosed}, and the fact that $\rho_\beta \wconv \rho$ in $\mX$, we obtain that
 \begin{eqnarray}\label{eq:carrier}
 	\rho\left(\overline{\bigcup_{\ve \re\gamma}\mV_{\gamma}\cap \mK_n}\right)
 	\geq
 	\limsup_{\beta \in \mE}\rho_\beta \left(\overline{\bigcup_{\ve \re
 	\gamma}\mV_{\gamma}\cap \mK_n}\right) \geq 1 - \delta_n, \quad \mbox{for all } \ve \in \mE \mbox{ and } n\in\mathbb N. 
 \end{eqnarray} 
 
Clearly, if $\ve_1 \re \ve_2$ then $\overline{\bigcup_{\ve_2\re \gamma}  \mV_{\gamma}\cap \mK_n} \subset \overline{\bigcup_{\ve_1 \re \gamma}  \mV_{\gamma}\cap \mK_n}$. Thus, for each $n\in \mathbb{N}$, the net $\left\{\overline{\bigcup_{\ve \re \gamma}  \mV_{\gamma}\cap \mK_n}\right\}_{\ve \in \mE}$ is a monotone decreasing net of compact sets in the compact Hausdorff space $\mK_n$. Moreover, since $\rho$ is a tight Borel probability measure on a Hausdorff space then $\rho$ is outer regular, see \cref{subsec:meas:theory}. Hence, \cref{contabove} applies, and we deduce that
\begin{align*}
	\rho\left(\limsup_{\ve \in \mE}(\mV_{\ve}\cap \mK_n)\right)
	= \rho\left(\bigcap_{\ve \in \mE} \overline{\bigcup_{\ve \re \gamma}  \mV_{\gamma}\cap \mK_n}\right) 
	= \lim_{\ve \in \mE} \rho\left(\overline{\bigcup_{\ve \re \gamma}  \mV_{\gamma}\cap \mK_n}\right) 
	\geq 1 - \delta_n
\end{align*}
for all $n \in \NN$.
 
Hence, from \eqref{Vcarrierofrho},
 \[ \rho(\mV) = \rho\left(\bigcup_{n=1}^\infty\limsup_\varepsilon (\mathcal{V}_\varepsilon\cap \mK_n)\right) \geq \rho\left(\limsup_{\ve \in \mE}(\mV_{\ve}\cap \mK_n)\right)\geq 1-\delta_n, \quad \mbox{ for all } n\in\mathbb N.
 \]
 Since $\delta_n \rightarrow 0$, by taking $n \to \infty$ we find that $\rho(\mV) = 1$. This completes the proof that the Borel set $\mV\subset \mU$ carries $\rho$ and that $\rho$ is a $\mU$-trajectory statistical solution.

 For the second part of the statement, let us now suppose that $I$ is closed and bounded on the left, with left endpoint $t_0$, and that $\Pi_{t_0}\rho_\ve \wconv \mu_0$. Since $\Pi_{t_0}: \mX \to X$ is a continuous mapping, the set $K_n = \Pi_{t_0}\mK_n$ is compact in $X$ and $\Pi_{t_0}K_n = \Pi_{t_0}^{-1}\Pi_{t_0}\mK_n \supset \mK_n$, so that
 \[
    \Pi_{t_0}\rho_\ve(X \setminus K_n) = \rho_\ve(\Pi_{t_0}^{-1}(X\setminus K_n)) = \rho_\ve(\mX \setminus \Pi_{t_0}^{-1} K_n) \leq \rho_\ve(\mX \setminus \mK_n) < \delta_n,
 \]
 showing that $\Pi_{t_0}\rho_\ve$ is tight on $X$. Similarly, $\Pi_{t_0}\rho$ is also tight on $X$. Let us show that $\Pi_{t_0}\rho_\ve \wconv \Pi_{t_0}\rho$ on $X$ using condition \ref{PMlimopen} of \cref{portmanteau}.
 
 First, for the whole space $X$, since $\Pi_{t_0}^{-1}X = \mX$, we have
\[ \Pi_{t_0}\rho_\ve(X) = \rho_\ve(\Pi_{t_0}^{-1}X) = \rho_\ve(\mX) \rightarrow \rho(\mX) = \rho(\Pi_{t_0}^{-1}X) = \Pi_{t_0}\rho(X).
\]
Now, for an open set $G\subset X$, the set $\Pi_{t_0}^{-1}G$ is open in $\mX$, so that
\[ \liminf_\ve \Pi_{t_0}\rho_\ve(G) = \liminf_\ve \rho_\ve(\Pi_{t_0}^{-1}(G)) \geq \rho(\Pi_{t_0}^{-1}(G)) = \Pi_{t_0}\rho(G).
\]
Thus, from \cref{portmanteau}, we deduce that $\Pi_{t_0}\rho_\ve \wconv \Pi_{t_0}\rho$ on $X$.

On the other hand, by assumption, we have $\Pi_{t_0}\rho_\ve \wconv \mu_0$. Since  $\MXt$ is Hausdorff with respect to the weak-star semicontinuity topology (see \cref{subsec:top:meas}), the weak-star semicontinuity limit has to be unique. Thus, we deduce that $\Pi_{t_0} \rho = \mu_0$. This concludes the proof.
\end{proof}

\begin{remark}
    The hypothesis, in \cref{Main1}, of the existence of compact subsets $\mK_n\subset\mX$ with $\rho_\ve(\mX \setminus \mK_n) < \delta_n$ and $\delta_n \rightarrow 0$ is, in fact, the condition of uniform tightness of the family $\{\rho_\ve\}_{\ve\in\mE}$, and this by itself guarantees, from the compactness given in \cref{topsoe}, the existence of the weak limit $\rho$. The importance of making this condition explicit in the statement of the theorem is only to relate the sets $\mK_n$ to the set $\mU$, via the condition that $\limsup_{\ve \in \mE}(\, \mV_{\ve}\cap \mK_n)\subset \mathcal U$, with respect to the topological superior limit in $\mX$. With this extra condition, we prove that the limit measure $\rho$ is carried by a Borel set included in $\mU$, and hence $\rho$ is a $\mU$-trajectory statistical solution. The localization of the carrier of $\rho$ is in fact the main point of \cref{Main1}, since the existence of a limit measure $\rho$ follows directly from the underlying uniform tightness condition.
\end{remark}

\begin{remark}
    From the proof of \cref{Main1}, we have, more precisely, that $\rho$ is carried by the $\sigma$-compact Borel subset $\mV$ defined in \eqref{Vcarrierofrho}, which depends on the sets $\mV_\ve$, $\ve \in \mE$, and also the sequence of compact sets $\mK_n\subset \mX$, $n \in \NN$. In this regard, we note that a carrier set is in general not unique.
\end{remark}

We turn to our second main result, relating to applications where the net of approximating trajectory statistical solutions is induced by a well-defined solution operator $S_\ve$, $\ve \in \mE$.

\begin{theorem}\label{Main2alt}
    Let $X$ be a Hausdorff space and let $I$ be an interval in $\mathbb R$ closed and bounded on the left with left endpoint $t_0$. Let $\mU$ be a subset of $\mX=\mathcal \Cloc(I,X)$. Consider a net $\{S_\ve\}_{\ve\in \mE}$ of measurable functions $S_{\ve}: X \rightarrow \mX$ and a net $\{\mu_\ve\}_{\ve\in\mE}$ of tight Borel probability measures on $X$. Set $\mU_\ve = S_\ve(X)$ and define $P_\ve = \Pi_{t_0}S_\ve:X\rightarrow X$. Assume that
	\begin{enumerate}[label=(H\arabic*)]
		\item \label{Main2:H0} $P_\ve\mu_\ve \wconv \mu_0$, for some tight Borel probability measure $\mu_0$ on $X$.
	\end{enumerate}    
	Moreover,  suppose that there exists a sequence $\{K_n\}_{n\in\NN}$ of compact sets in $X$ with the following properties:
    \begin{enumerate}[label=(H\arabic*)] \setcounter{enumi}{1}
        \item \label{Main2:H1} The map $S_{\ve}|_{K_n}: K_n\rightarrow \mX$ is continuous for all $n\in\NN$ and all $\varepsilon\in \mathcal E$, with the topology inherited from $X$;
        \item \label{Main2:H2} $\mu_\ve(X\setminus K_n) \leq \delta_n$, for all $\ve\in\mE$ and all $n\in\NN$, where $\delta_n \rightarrow 0,$ when $n\rightarrow \infty$;
        \item \label{Main2:H3} For each $n$, there exists a compact set $\mK_n$ in $\mX$ with $S_\ve(K_n) \subset \mK_n$, for all $\ve\in\mE$.
        \item \label{Main2:H4} For each $n\in\NN$,
        \[\limsup_{\ve \in \mE}S_{\ve}(K_n) \subset \mathcal U.\]
    \end{enumerate}
    Then, each $\rho_\ve=S_\ve \mu_\ve$ is a $\mU_\ve$-trajectory statistical solution and the net $\{\rho_{\ve}\}_{\ve \in \mE}$ has a convergent subnet, with respect to the weak-star semicontinuity topology, to a $\mathcal U$-trajectory statistical solution $\rho$ such that $\Pi_{t_0}\rho=\mu_0$.
\end{theorem}

\begin{remark}
    Note that, given $\{\bu_{0, \ve}\}_{\ve \in \mE} \subset X$, the family $P_\ve \bu_{0,\ve} = \Pi_{t_0} S_\ve \bu_{0,\ve}$ represents a collection of approximating initial data associated with the solution operators $\{S_\ve\}_{\ve\in\mE}$ in a given application. Hence, condition \ref{Main2:H0} is a natural assumption that guarantees that the initial measures $\Pi_{t_0} S_\ve \mu_\ve$ for the approximations converge to the limit initial measure. This should be checked in each application. In some cases, one has simply $\Pi_{t_0} S_\ve = I$, i.e. the identity operator, as in the application to the inviscid limit of the 2D Navier-Stokes equations given in \cref{subsec:2D:NSE:Euler}. However, this is not the case in certain applications, for example, when the approximating systems evolve in some lower-dimensional approximation of $X$, as in the application to Galerkin approximations of the Navier-Stokes equations presented in \cref{subsec:3D:Galerkin:NSE}. Indeed, in the Galerkin case we have $\Pi_{t_0} S_\ve u= P_N u$, $u \in X$, with $P_N$ denoting the projection onto the space spanned by $N \in \NN$ basis vectors of the vector space $X$.
\end{remark}

\begin{remark}
    Condition \ref{Main2:H1} pertains to the regularity of the approximation semigroup and is used for measurability purposes, guaranteeing that $\rho_\ve$ is carried by a Borel subset of $\mU_\ve$ and, hence, is a statistical solution. Condition \ref{Main2:H2} guarantees that the initial measures $\mu_\ve$ are uniformly exhausted by the sets $K_n$, which is then used together with \ref{Main2:H3} to prove the uniform tightness of $\rho_\ve = S_\ve\mu_\ve$ with respect to the compact sets $\{\mK_n\}$. Condition \ref{Main2:H4} certifies that the family of operators are, in fact, an approximation of the limit problem, i.e. the approximations do converge to a solution of the limit problem, as expressed by the space $\mU$, so that the limit measure $\rho$ is a $\mU$-statistical solution.
\end{remark}

\begin{remark}
In the particular case where $P_\ve$ is the identity operator and $\mu_\ve=\mu_0$ for all $\ve\in\mE$, for some $\mu_0$ tight Borel probability measure, assumptions \ref{Main2:H0} and \ref{Main2:H2} of \cref{Main2alt} are immediately satisfied.  Several applications would fit into this setting. This is indeed the case in our application to the inviscid limit of the 2D Navier-Stokes equations in \cref{subsec:2D:NSE:Euler}.
\end{remark}

\begin{proof}[Proof of \cref{Main2alt}]
    From the definition of $\rho_\ve$ as $S_\ve\mu_\ve$, we have
    \[
        \rho_\ve(\mX \setminus S_\ve(K_n)) = \mu_\ve(S_\ve^{-1}(\mX \setminus S_\ve(K_n))) = \mu_\ve(X \setminus S_\ve^{-1}(S_\ve(K_n))).
    \]
    Since $S_\ve^{-1}(S_\ve(K_n)) \supset K_n$, it follows that $X \setminus S_\ve^{-1}(S_\ve(K_n)) \subset X \setminus K_n$, so that
    \[
        \rho_\ve(\mX \setminus S_\ve(K_n)) \leq \mu_\ve(X\setminus K_n).
    \]
    Using \ref{Main2:H2}, we obtain
    \begin{equation}
        \label{eqfrommain2h3}
        \rho_\ve(\mX \setminus S_\ve(K_n)) \leq \delta_n,
    \end{equation}
    for all $n\in \NN$ and arbitrary $\varepsilon$. Thus,
    \[
        \rho_\ve(\mX \setminus \cup_{m\in\NN} S_\ve(K_m)) \leq \rho_\ve(\mX \setminus S_\ve(K_n)) \leq \delta_n,
    \]
    for all $n$. Since $\delta_n\rightarrow 0$, this implies that
\[
\rho_\ve(\mX \setminus \cup_{n\in\NN} S_\ve(K_n)) = 0.
\]
    
    Therefore, $\rho_\ve$ is carried by the set
    \begin{equation}
        \label{defmVneforSve}
        \mV_\ve = \bigcup_{n\in\NN} S_\ve(K_n).
    \end{equation}
    From \ref{Main2:H1} we see that $S_\ve(K_n)$ is compact in $\mX$, so that $\mV_\ve$ is a $\sigma$-compact set in $\mX$. In particular, it is a Borel set in $\mX$.
    
    Let us now show that each $\rho_\ve$ is tight. Using \ref{Main2:H2}, we see that, for every Borel set $\mA \subset \mX$,
    \begin{align*}
        \rho_\ve (\mA) & = \mu_\ve(S_\ve^{-1}(\mA)) \\
            & = \mu_\ve(S_\ve^{-1}(\mA) \cap K_n) + \mu_\ve(S_\ve^{-1}(\mA) \setminus K_n) \\
            & \leq \mu_\ve(S_\ve^{-1}(\mA) \cap K_n) + \mu_\ve(X \setminus K_n) \\
            & \leq \mu_\ve(S_\ve^{-1}(\mA) \cap K_n) + \delta_n.
    \end{align*}
    On the other hand,
    \[
        \mu_\ve(S_\ve^{-1}(\mA) \cap K_n) \leq \mu_\ve(S_\ve^{-1}(\mA)) = \rho_\ve (\mA).
    \]
    Thus,
    \begin{equation}
        \label{rhonemAassupn}
        \rho_\ve (\mA) = \sup_n\{\mu_\ve (S_\ve^{-1}(\mA)\cap K_n)\}.
    \end{equation}
    Being tight, each $\mu_\ve$ is continuous from below with respect to compact sets. Thus,
    \[
        \mu_\ve (S_\ve^{-1}(\mA)\cap K_n) = \sup\{\mu_\ve(F): F \subset S_\ve^{-1}(\mA)\cap K_n, F \mbox{ compact in } X\}.
    \]
    Since $F \subset S_\ve^{-1}(S_\ve(F))$, we have the bound
    \[
        \mu_\ve (S_\ve^{-1}(\mA)\cap K_n) \leq \sup\{\mu_\ve(S_\ve^{-1}(S_\ve(F))): F \subset S_\ve^{-1}(\mA)\cap K_n, F \mbox{ compact in } X\}.
    \]
    From the continuity of the restriction $S_\ve |_{K_n}$ of $S_\ve$ to the compact set $K_n$, we have $S_\ve(F)$ compact, for every compact set $F\subset K_n$. Thus, we bound the right hand side extending $S_\ve(F)$ to any compact set $\mK'\subset \mA$, i.e.
    \[
        \mu_\ve (S_\ve^{-1}(\mA)\cap K_n) \leq \sup\{\mu_\ve (S_\ve^{-1}(\mK')): \mK' \subset \mA, \,\mK' \mbox{ compact in } \mX\}.
    \]
    Back to $\rho_\ve = S_\ve\mu_\ve$, this means
    \[
        \mu_\ve (S_\ve^{-1}(\mA)\cap K_n) \leq \sup\{\rho_\ve(\mK'): \mK' \subset \mA, \,\mK' \mbox{ compact in } \mX\}.
    \]
    Since the right hand side does not depend on $n$, this gives
    \[
        \sup_n\mu_\ve (S_\ve^{-1}(\mA)\cap K_n) \leq \sup\{\rho_\ve(\mK'): \mK' \subset \mA, \,\mK' \mbox{ compact in } \mX\}.
    \]
    Plugging this back into \eqref{rhonemAassupn} yields
    \[
        \rho_\ve (\mA) \leq \sup_n\{\mu_\ve (S_\ve^{-1}(\mA)\cap K_n)\} \leq \sup\{\rho_\ve(\mK'): \mK' \subset \mA, \,\mK' \mbox{ compact in } \mX\} \leq \rho_\ve(\mA).
    \]
    In other words,
    \[
        \rho_\ve (\mA) = \sup\{\rho_\ve(\mK'): \mK' \subset \mA, \,\mK' \mbox{ compact in } \mX\}.
    \]
    This shows that $\rho_\ve$ is tight. Thus, $\rho_\ve$ is a tight measure carried by the Borel subset $\mV_\ve$ of $\mU_\ve$, which means that $\rho_\ve$ is a $\mU_\ve$-statistical solution.

    Now, using \ref{Main2:H3} and the previous estimate \eqref{eqfrommain2h3}, we see that
    \[
        \rho_\ve(\mX \setminus \mK_n) \leq \rho_\ve(\mX \setminus S_\ve(K_n)) \leq \delta_n.
    \]
    This implies that the family $\{\rho_\ve\}_{\ve\in\mE}$ is uniformly tight.
    
    Since $\{\rho_\ve\}_{\ve\in\mE}$ is uniformly tight, it follows from \cref{topsoe} that there exists a subnet which converges in the weak-star semicontinuity topology to a tight probability measure $\rho \in \MmXt$.
    Now we show directly that $\rho$ is carried by
    \[
        \mV = \bigcup_{n=1}^\infty \limsup_{\ve \in \mE} S_\ve(K_n).
    \]
    Note that, by the definition \eqref{def:limsup} of $\limsup_\ve$ as an intersection of closed sets, we have that $\limsup_\ve S_\ve(K_n)$
    is a Borel set and so is $\mV$. Moreover, by assumption \ref{Main2:H4}, each $\limsup_\ve S_\ve(K_n)$ is a subset of $\mU$, hence $\mV \subset \mU$ as well. Let us now mimic the proof in \cref{Main1} and show that $\rho$ is carried by $\mV$.

    From the estimate \eqref{eqfrommain2h3}, we find that
    \[
        \rho_\ve(S_\ve(K_n)) \geq 1 - \delta_n.
    \]
    For every $\beta\in\mE$, there exists $\alpha\in\mE$ such that $\beta \preccurlyeq \alpha$ and $\ve \preccurlyeq \alpha$. Thus,
    \[
        \rho_\alpha\left( \overline{\bigcup_{\ve \preccurlyeq \gamma} S_\gamma(K_n)}\right) \geq \rho_\alpha(S_\alpha(K_n)) \geq 1 - \delta_n,
    \]
    for arbitrary $n\in\mathbb{N}.$ Taking the supremum over $\alpha\in \mE$, for $\beta \preccurlyeq \alpha$,
    \[
        \sup_{\alpha\in\mE, \;\beta \preccurlyeq \alpha} \rho_\alpha\left(\overline{\bigcup_{\ve \preccurlyeq \gamma} S_\gamma(K_n)}\right) \geq \sup_{\alpha\in\mE, \;\beta \preccurlyeq \alpha, \ve \preccurlyeq \alpha} \rho_\alpha\left(\overline{\bigcup_{\ve \preccurlyeq \gamma} S_\gamma(K_n)}\right) \geq 1 - \delta_n.
    \]
    Taking, now, the infimum over $\beta\in\mE$,
    \[
        \limsup_{\beta\in\mE} \rho_\beta\left(\overline{\bigcup_{\ve \preccurlyeq \gamma} S_\gamma(K_n)}\right) = \inf_{\beta\in\mE}\sup_{\alpha\in\mE, \;\beta \preccurlyeq \alpha}\rho_\alpha\left(\overline{\bigcup_{\ve \preccurlyeq \gamma} S_\gamma(K_n)}\right) \geq 1 - \delta_n.
    \]
    Since $\rho_\beta \wconv \rho$ along a subnet $\beta\in \mE'$, it follows from \cref{portmanteau}, \ref{PMclosed}, that
    \[
        \rho\left(\overline{\bigcup_{\ve \preccurlyeq \gamma} S_\gamma(K_n)}\right) \geq \limsup_{\beta\in\mE'} \rho_\beta\left(\overline{\bigcup_{\ve \preccurlyeq \gamma} S_\gamma(K_n)}\right) \geq 1 - \delta_n.
    \]
    Each $S_\gamma(K_n)$ is included in the compact set $\mK_n$, so that the set $\overline{\bigcup_{\ve \preccurlyeq \gamma} S_\gamma(K_n)}$ is a closed set in $\mK_n$, hence compact. Moreover, the family of sets $\overline{\bigcup_{\ve \preccurlyeq \gamma} S_\gamma(K_n)}$, $\ve\in\mE$, is decreasing in $\ve$. Hence, it follows from \cref{contabove} that
    \[
        \rho\left(\limsup_{\ve\in\mE} S_\ve(K_n)\right) = \rho\left( \bigcap_{\ve\in\mE}\overline{\bigcup_{\ve \preccurlyeq \gamma} S_\gamma(K_n)}\right) = \lim_{\ve\in\mE}\rho\left(\overline{\bigcup_{\ve \preccurlyeq \gamma} S_\gamma(K_n)}\right) \geq 1 - \delta_n.
    \]
    Therefore,
    \[
        \rho(\mV) = \rho\left( \bigcup_{m\in\NN} \limsup_{\ve\in\mE} S_\ve(K_m)\right) \geq 1 - \delta_n.
    \]
    Finally, since this holds for any $n\in\NN$ and $\delta_n\rightarrow 0$, we find that
    \[
        \rho(\mV) = 1.
    \]
    Thus, we find a tight Borel probability measure $\rho$ with $\rho_\ve \wconv \rho$ along a subnet $\ve\in\mE'$, with $\rho$ carried by the Borel subset $\mV$ of $\mU$, which means that $\rho$ is a $\mU$-trajectory statistical solution.

    It remains to show that $\Pi_{t_0} \rho = \mu_0$. We have just proved that $\rho_\ve \wconv \rho$ as measures on $\mX$. Since $\Pi_{t_0}$ is continuous from $\mX$ to $X$, this implies, as seen in the proof of \cref{Main1}, that $\Pi_{t_0}\rho_\ve \wconv \Pi_{t_0}\rho$, and $\Pi_{t_0} \rho_\ve, \Pi_{t_0} \rho \in \MXt$. 
    
    On the other hand, from hypothesis \ref{Main2:H0}, we obtain
    \[ \Pi_{t_0}\rho_\ve = \Pi_{t_0}S_\ve\mu_\ve = P_\ve\mu_\ve \wconv \mu_0.
    \]
    Combining the two limits together and invoking the uniqueness of the weak-star semicontinuity limit in $\MXt$, it follows that $\Pi_{t_0}\rho = \mu_0$, which completes the proof.
\end{proof}

\begin{remark}
    \label{rmkthm2vsthm1}
    Notice we do not apply \cref{Main1} to prove \cref{Main2alt}. If instead we assume $\limsup_{\ve \in \mE} (\mU_\ve \cap \mK_n) \subset \mU$ in place of \ref{Main2:H4}, then the hypothesis \ref{ii:Main1} of \cref{Main1} is satisfied. However, condition \ref{Main2:H4} is weaker and, in fact, more natural in this context. For this reason, we prove hypothesis \ref{i:Main1} of \cref{Main1} and then mimic the remaining part of the proof of \cref{Main1} to complete the proof \cref{Main2alt}.
\end{remark}

\section{Applications}\label{secapplications}

This section provides applications of our general framework for convergence of statistical solutions from \cref{Main1} and \cref{Main2alt} above. \cref{subsec:2D:NSE:Euler} and \cref{subsec:3D:NSE:Euler} concern the inviscid limit of the Navier-Stokes equations towards the Euler equations in two and three dimensions, respectively. And \cref{subsec:3D:Galerkin:NSE} deals with spectral Galerkin discretizations approximating the 3D Navier-Stokes equations. Before delving into these applications, we first recall in \cref{subsec:setting} some preliminary background regarding the Euler and Navier-Stokes equations.

\subsection{Mathematical setting for 2D and 3D incompressible flows}\label{subsec:setting}

We consider the $d$-dimensional incompressible Navier-Stokes equations (NSE) for either $d = 2$ or $3$, given by
\begin{align}\label{nse}
\partial_t \bu-\nu\Delta\bu+(\bu\cdot\nabla)\bu+\nabla p=\mathbf{f}, \quad \nabla\cdot\bu=0,
\end{align}
where $\bu=(u_1,\ldots, u_d)$ and $p$ are the unknowns and represent the velocity field and the pressure, respectively. Moreover, $\f=(f_1,\ldots, f_d)$ represents
a given body force applied to the fluid and $\nu>0$ is the kinematic viscosity. 
The functions $\bu$, $p$ and $\f$ depend on a spatial variable $x$ varying in $\Omega \subset \mathbb{R}^d$ and on a time variable $t$ varying in an interval $I \subset \mathbb{R}$. We will refer to  \eqref{nse} as `$\nu$-NSE' whenever there is a need to emphasize the dependence on $\nu$.

In the inviscid case, i.e. when $\nu = 0$, \eqref{nse} becomes the $d$-dimensional incompressible Euler equations
\begin{align}\label{ee}
	\partial_t \bu+(\bu\cdot\nabla)\bu+\nabla p=\mathbf{f}, \quad \nabla\cdot\bu=0.
\end{align}

We assume for simplicity that \eqref{nse} and \eqref{ee} are subject to periodic boundary conditions, with $\Omega = (0, L_1) \times \ldots \times (0, L_d) \subset \RR^d$ denoting a basic domain of periodicity. We say that a function $\bv: \mathbb{R}^d \to \mathbb{R}^d$ 
is $\Omega$-periodic if $\bv$ is periodic with period $L_i$ in each spatial direction $x_i$, $i=1, \ldots, d$.

Let us fix the functional setting associated to these equations. Denote by $\mC^{\infty}_{per}(\Omega)^d$ the space of infinitely differentiable and $\Omega$-periodic functions defined on $\RR^d$, and let $\Tsp \subset \mC^{\infty}_{per}(\Omega)^d$ be the set of divergence-free and periodic test functions with vanishing spatial average, namely
\begin{align}\label{def:Tsp}
	\Tsp=\left\{\bu\in \mC^{\infty}_{per}(\Omega)^d \,:\, \nabla \cdot \bu=0\mbox{ and }\int_\Omega \bu(\bx)\;\rd\bx=0 \right\}.
\end{align}

We denote by $H$, $V$ and $W_\sigma^{1,r}$, $1\leq r\leq \infty$, the closures of $\Tsp$ with respect to the norms in $L^2(\Omega)^d$, $H^1(\Omega)^d$ and $W^{1,r}(\Omega)^d$, respectively. Note that $V=W^{1,2}_\sigma$. The inner product and norm in $H$ are defined, respectively, by
\[
(\bu,\bv)=\int_{\Omega}\bu\cdot \bv \;\rd\bx \quad \mbox{ and } \quad |\bu| = \sqrt{(\bu,\bu)},
\]
where $\bu\cdot\bv = \sum_{i=1}^3 u_iv_i$. In the space $V$, these are defined as
\[
(\!(\bu, \bv)\!) = (\nabla \bu,\nabla \bv) = \int_\Omega \nabla \bu : \nabla \bv \;\rd \bx \quad \mbox{ and } \quad \|\bu\| = \sqrt{(\!(\bu,\bu)\!)},
\]
where it is understood that $\nabla \bu=(\partial_{x_j} u_i)_{i,j=1}^d$ and that $\nabla \bu : \nabla \bv$ is the componentwise product
between $\nabla \bu$ and $\nabla \bv$. In the space $\Wr$, except for $r=2$, we can only define a norm, given by
\[\|\bu\|_{W_\sigma^ {1,r}} = \left\{ \begin{array}{ll}
\left( \displaystyle \sum_{i,j=1}^d\| \partial_{x_j} u_i\|_{L^r}^r\right)^{\frac{1}{r}}, & 1\leq r <\infty\\
\\
\displaystyle \sum_{i,j=1}^d\|\partial_{x_j} u_i\|_{L^r}, & r=\infty . \end{array}\right. 
\]
The fact that $\| \cdot\|$ and $\|\cdot \|_{\Wr}$, $1 \leq r \leq \infty$, are indeed norms follows from the Poincar\'e inequality \eqref{ineq:Poincare} and the inequality \eqref{ineq:Wr:V} below.

We denote by $H'$, $V'$ and $\Wrp$ the dual spaces of $H$, $V$ and $W_\sigma^{1,r}$, respectively. The dual spaces are endowed with the classical dual norm of Banach spaces. Namely, for a given Banach space $E$, the standard norm in the dual space $E'$ is given by $\|\bu\|_{E'}=\sup_{\|\bv\|_E\leq 1}\langle \bu,\bv\rangle_{E',E}$, where $\langle \cdot, \cdot \rangle_{E',E}$ denotes the duality product between $E$ and $E'$. After identifying $H$ with its
dual $H'$, we obtain $V \subset H \subset V'$ and $W_\sigma^{1,r} \subset H \subset \Wrp$, with the injections being continuous, compact, and 
each space dense in the following one. Also, since $\Omega$ is bounded, we have that $W^{1,r}_\sigma\subset V$, with continuous injection, for all $r\geq 2$.

The negative Laplacian operator $(-\Delta)$ on $V \cap H^2(\Omega)^d$ is a positive and self-adjoint operator with compact inverse. As such, it admits a nondecreasing sequence of positive eigenvalues $\{\lambda_k\}_{k \in \NN}$ with $\lambda_k \to \infty$ as $k \to \infty$, which is associated to a sequence of eigenfunctions $\{\bw_k\}_{k \in \NN}$ that consists of an orthonormal basis of $H$. In relation to the first eigenvalue $\lambda_1$ of $(-\Delta)$, we have the Poincar\'e inequality,
\begin{align}\label{ineq:Poincare}
	| \bu | \leq \lambda_1^{-1/2} \|\bu\|, \quad \mbox{ for all } \bu \in V.
\end{align}
Using H\"older's inequality, we have
\begin{align}\label{ineq:Wr:V:volume}
	\| \bu \| \leq |\Omega|^{\left( \frac{1}{2} - \frac{1}{r} \right)} \| \bu \|_{\Wr}, \quad \mbox{ for all } \bu \in \Wr, \,\, 2 \leq r \leq \infty,
\end{align}
where $|\Omega| = L_1 \cdots L_d$ is the area or volume of the $d$-dimensional domain. For the sake of simplicity, and with the aim of using $\lambda_1$ for dimensional consistency, we write \eqref{ineq:Wr:V:volume} in terms of $\lambda_1,$ by introducing the non-dimensional constant $c = \max\{1, |\Omega|\lambda_1^{d/2}\}^{1/2},$ so that
\begin{align}\label{ineq:Wr:V}
	\| \bu \| \leq c \lambda_1^{-\frac{d}{2} \left( \frac{1}{2} - \frac{1}{r} \right)} \| \bu \|_{\Wr}, \quad \mbox{ for all } \bu \in \Wr, \,\, 2 \leq r \leq \infty.
\end{align}

For any normed space $E$, we denote by $B_{E}(R)$ the closed ball centered at $0$ and with radius $R>0$ in $E$. Moreover, we denote by $E_{\text{w}}$ and $B_{E}(R)_{\text{w}}$ the spaces $E$ and $B_{E}(R)$ endowed with the weak topology, respectively.

For $d=2$, we denote by $\nabla^\perp$ the operator defined as $(-\partial_{x_2},\partial_{x_1})$. The vorticity $\omega=\omega(x_1,x_2)$ associated with a velocity field $\bu(x_1,x_2)=(\bu_1(x_1,x_2),\bu_2(x_1,x_2))$ is given by
\[ \omega=\nabla^\perp \cdot \bu = -\partial_{x_2} \bu_1 + \partial_{x_1} \bu_2.\]

We recall that the $L^r$-norm of the vorticity controls the $\Wr$-norm of its associated velocity field. More precisely, let 
$\bu\in H$ be such that $\nabla^{\perp}\cdot \bu\in L^r(\Omega)$, for some $1<r<\infty$. Then $\bu\in \Wr$ and
\begin{equation}\label{divcurl}
\|\bu\|_{\Wr}\leq c\|\nabla^{\perp}\cdot\bu\|_{L^r},
\end{equation}
for some other positive constant $c$.

We don't need to track the different constants that appear in the estimates, so, in what follows, we denote as $c$ a dimensionless positive constant whose value may change from line to line. We also occasionally use the capital letter $C$ to denote a positive dimensional constant.

Before proceeding to \cref{subsubsec:2D} and \cref{subsubsec:3D} with the types of solutions of NSE and Euler that suit our purposes, let us briefly provide some context for the choice of such solutions and recall some of the currently available results on existence and uniqueness. We keep the discussion restricted to the case of periodic boundary conditions, although similar results are often valid with other types of boundary conditions. We refer to the references cited below for further details.

Regarding the NSE in two dimensions, it is well known that given any forcing term $\f \in L^2_{\textrm{loc}} ([t_0,\infty), V')$, for some $t_0 \geq 0$, and initial datum $\bu_0$ in $H,$ there exists a unique weak solution $\bu$ of \eqref{nse} on $[t_0,\infty)$ satisfying $\bu(t_0) = \bu_0$, see e.g. \cite{bookcf1988,Lady,temam84}. Here the exact meaning of ``weak solution'' is recalled in \cref{NSweak} below. Therefore, the initial-value problem for weak solutions of the 2D $\nu$-NSE is globally well-posed, and we may thus define a solution operator $S_\nu$ associating to each $\bu_0 \in H$ the corresponding unique solution $\bu$ of \eqref{nse} on $[t_0,\infty)$ satisfying $\bu(t_0) = \bu_0$.

In the three-dimensional case, it is known that for any given $\f \in L^2_{\textrm{loc}} ([t_0,\infty), V')$ and $\bu_0 \in H$ there exists a \emph{Leray-Hopf weak solution} of NSE (cf. \cref{3NSweak} below) on $[t_0,\infty)$ satisfying the initial condition $\bu(t_0) = \bu_0$. This solution is typically obtained as an appropriate limit of the unique solutions of a corresponding sequence of approximating Galerkin systems, see e.g. \cite{bookcf1988,Lady,temam84}. However, this Leray-Hopf weak solution is not currently known to be unique, and hence a corresponding solution operator cannot be defined as of yet. Regarding this uniqueness issue, it is worth pointing out the recent result in \cite{AlbrittonBrueColombo2022} where the authors show that for a suitably constructed non-smooth forcing function $\f$ there exist two distinct Leray-Hopf weak solutions of 3D NSE in $\RR^3 \times (0,T)$ with initial data $\bu_0 \equiv 0$, for some $T>0$. Additionally, non-uniqueness results for weak solutions of 3D NSE of non-Leray-Hopf type were proved in \cite{BuckmasterVicol2019,BuckmasterColomboVicol2020,CheskidovLuo2022} by using convex integration techniques.

For our applications, we focus on the notions of weak solutions of 2D NSE and Leray-Hopf weak solutions of 3D NSE as recalled in \cref{NSweak} and \cref{3NSweak} below, respectively. In view of the aforementioned results, the examples showing convergence of statistical solutions in the 2D inviscid limit (\cref{subsec:2D:NSE:Euler}) and for Galerkin approximations in 3D (\cref{subsec:3D:Galerkin:NSE}) follow as a consequence of \cref{Main2alt}. The 3D inviscid limit case (\cref{subsec:3D:NSE:Euler}), on the other hand, requires the setting of \cref{Main1}. 

Specifically, for the Galerkin application, we take each $S_\ve$ from \cref{Main2alt} to be the solution operator $S_N$ for the Galerkin system with $N \in \NN$ Galerkin modes (see \eqref{def:Sm:Gal}), and $\mU$ as the set $\mU^\nu_I$ of Leray-Hopf weak solutions of 3D $\nu$-NSE on a fixed time interval $I \subset \RR$. 

Regarding the inviscid limit examples, in the 2D case we take each $S_\ve$ from \cref{Main2alt} as the solution operator $S_\nu$ associated to the 2D Navier-Stokes equations with viscosity parameter $\nu > 0$ (see \eqref{def:Snu}). In the 3D inviscid limit case, we consider each set $\mU_\varepsilon$ from \cref{Main1} to be the family $\mU^\nu_I$ of Leray-Hopf weak solutions of 3D $\nu$-NSE on the time interval $I \subset \RR$, and $\rho_\varepsilon = \rho_\nu$ as a corresponding trajectory statistical solution in the sense of \cref{def-stat-sol}. We note that existence of such trajectory statistical solution $\rho_\nu$ in 3D satisfying a given initial condition $\Pi_{t_0} \rho_\nu = \mu_0$, for any Borel probability measure $\mu_0$ on $H$, follows from the work \cite{FRT2013}, but is also obtained in \cite[Theorem 4.2]{BMR2016} with a more streamlined proof. Additionally, as pointed out in \cref{rmk:exist:ss:3D:NSE} below, this existence result also follows via convergence of statistical solutions of corresponding Galerkin approximations, as a consequence of our application in \cref{subsec:3D:Galerkin:NSE}.

To complete the setup for the 2D and 3D inviscid limit applications as required from \cref{Main1} and \cref{Main2alt}, respectively, it remains to choose an appropriate set $\mU$ of solutions of the Euler equations \eqref{ee}. In view of assumption \ref{ii:Main1} from \cref{Main1} or \ref{Main2:H4} in \cref{Main2alt}, we must choose a set $\mU$ for which it holds that any vanishing viscosity convergent sequence of individual solutions $\bu_{\nu_j}$ in $\mU_{\nu_j}$, $j \in \NN$, lying in a certain compact set, has as its limit a solution in $\mU$. In the 2D periodic case, this inviscid limit result for individual solutions is known to hold with respect to the standard notions of weak solutions of NSE and Euler (cf. \cref{Eweak}), provided enough regularity is assumed on the initial data. More precisely, given any $\bu_0 \in H$ such that $\nabla^\perp \cdot \bu_0 \in L^r$, $1 < r \leq \infty$, a vanishing viscosity convergent sequence of weak solutions to the 2D NSE, each with initial datum $\bu_0$, has as its limit a weak solution of 2D Euler with the same initial datum. This is shown in \cite{Lions2013} (see also \cite{DiPernaMajda1987}) under the assumption of zero forcing term, but it is mentioned that more general forcing terms could also be considered. See \cref{thmNStoEuler} below for one such more general case.

Here we recall that in the case $r = \infty$, namely when $\nabla^\perp \cdot \bu_0 \in L^\infty$, there is \emph{at most one} weak solution $\bu$ to the 2D Euler equations in vorticity formulation on $[t_0,\infty)$ satisfying $\bu(t_0) = \bu_0$, as originally shown in \cite{Yudovich1963}. This implies that, given any Borel probability measure $\mu_0$ on $H$ that is carried by the set $\mathcal{O}_0 = \{\bu_0 \in H \,:\, \nabla^\perp \cdot \bu_0 \in L^\infty\}$, a trajectory statistical solution of the 2D Euler equations starting from this initial measure can be simply obtained as $S \mu_0$, where $S$ is an associated and well-defined solution operator on $\mathcal{O}_0$. Moreover, together with the inviscid limit result for individual solutions, one can easily establish the convergence $S_{\nu_j} \mu_0 \wconv S \mu_0$ for any sequence $\nu_j \to 0$, where as before $S_\nu$ denotes the solution operator associated to the 2D $\nu$-NSE. For this reason, in our results below in \cref{subsec:2D:NSE:Euler} we consider only $r < \infty$.

In the three-dimensional case, on the other hand, an analogous inviscid result for individual solutions as previously described is not currently available. 
Alternative, and weaker, definitions of solutions for the Euler equations were defined to circumvent the extra complications that arise in three dimensions, and consequently obtain existence of a certain type of global-in-time solution of Euler as a vanishing viscosity limit, under appropriate initial data. Two such weaker notions are the \emph{measure-valued solutions} proposed in \cite{DiPernaMajda1987} and the \emph{dissipative solutions} from \cite{Lions2013}. As mentioned in \cref{sec:Intro}, here we focus on the latter definition (cf. \cref{Edissipative} below), since it more directly fits our abstract framework from \cref{secAbstractRes}. 

Finally, it is worth mentioning that global existence of weak solutions to the Euler equations from any given initial datum in $H$ and for any dimension $d \geq 2$ was recently shown in \cite{Wiedemann2011}, but not as a vanishing viscosity limit. Specifically, \cite{Wiedemann2011} relies on the construction of ``wild'' solutions of the Euler equations developed in \cite{DLS2010} to show the existence of an infinite number of (wild) weak solutions of \eqref{ee} departing from any fixed initial datum in $H$, and under zero forcing term.

\subsubsection{2D incompressible flows}\label{subsubsec:2D}

Let $I\subset\mathbb{R}$ be an interval closed and bounded on the left with left endpoint $t_0$. We start by recalling the standard notions of weak solutions to the 2D Navier-Stokes and Euler equations on $I$. For the definitions below, we recall the space $\Tsp$ of test functions defined in \eqref{def:Tsp}.

\begin{definition}\label{NSweak}
	Let $\f\in L_{\rloc}^2(I, V')$. We say that $\bu$ is a \textbf{weak solution of the 2D Navier-Stokes equations, \eqref{nse}, on $I$} if
\begin{enumerate}[label={(\roman*)}]
 \item $\bu\in \mathcal C_{\rloc}(I,H)\cap L_{\rloc}^2(I, V)$;
 \item $\partial_t\bu \in   L_{\rloc}^{2}(I, V')$;
 \item \label{NSweakeq} For every $\bv \in \Tsp$, the equation
 \begin{align}\label{eq3b}
 	\frac{d}{dt} \int_\Omega \bu \cdot \bv \; \rd \bx + \nu \int_\Omega \nabla \bu : \nabla \bv \; \rd \bx - \int_\Omega \bu \otimes \bu : \nabla \bv \; \rd \bx = \langle \f, \bv \rangle_{V',V} 
 \end{align}
 is satisfied in the sense of distributions on $I$.
\end{enumerate}
\end{definition}

\begin{definition}\label{Eweak}
	Let $\f\in L_{\rloc}^2(I, V')$. We say that $\bu$ is a \textbf{weak solution of the 2D Euler equations, \eqref{ee}, on $I$} if
\begin{enumerate}[label={(\roman*)}]
 \item\label{Eweaki} $\bu\in L_{\rloc}^\infty(I,H) \cap \mC_{\rloc}(I,\Tsp')$;
 \item \label{Eweakii} For every $\bv \in \Tsp$, the equation
 \begin{align}\label{Eweakiiineq}
 	\frac{d}{dt} \int_\Omega \bu \cdot \bv \; \rd \bx  - \int_\Omega \bu \otimes \bu : \nabla \bv \; \rd \bx = \langle \f, \bv \rangle_{V',V} 
 \end{align}
 is satisfied in the sense of distributions on $I$.
\end{enumerate}
\end{definition}

As recalled in \cref{subsec:setting}, when given $\bu_0 \in H$ with $\nabla^\perp \cdot \bu_0  \in L^r(\Omega)$, for some $1 < r \leq \infty$, the existence of a weak solution $\bu$ to the 2D Euler equations on $I$ satisfying $\bu(t_0) = \bu_0$ in $\Tsp'$ can be shown via a vanishing viscosity limit. A proof is given in e.g. \cite[Theorem 4.1]{Lions2013} in the case of zero forcing term. However, it is not difficult to extend the proof to the case of nonzero forcing $\f$, by assuming that $\f$ satisfies, e.g., $\f \in L_{\rloc}^2(I, H)$ and
$\nabla^\perp \cdot \f \in L^r_{\rloc}(I,L^r(\Omega))$. Additionally, it follows from the proof that this weak solution $\bu$ also belongs to $\mC_{\rloc}(I,\Wrw)$. 

Moreover, under these same conditions on $\bu_0$ and $\f$, it is not difficult to verify that the corresponding unique weak solution $\bunu$ of the 2D Navier-Stokes equations on $I$ satisfying $\bunu(t_0) = \bu_0$ in $H$ also belongs to $\mC_{\rloc}(I,\Wrw)$. 

In our following results regarding the two-dimensional case, we shall maintain this assumption on $\f$, namely $\f \in L_{\rloc}^2(I, H)$ and $\nabla^\perp \cdot \f \in L^r_{\rloc}(I,L^r(\Omega))$. In this case, it thus follows that we may take the abstract space $X$ from \cref{secAbstractRes} as $\Wrw$, with the corresponding trajectories of Euler and Navier-Stokes lying in $\mC_{\rloc}(I,\Wrw)$.

In the proposition below, we collect some useful inequalities valid for weak solutions of the 2D NSE, \eqref{nse}. The proof follows with similar arguments from \cite[Section 4.1]{Lions2013} under the appropriate modifications to include the forcing term $\f$. We omit the details.

Here we point out that the upper bound in \eqref{estapriori2} below is uniformly bounded as $\nu \to 0$. Clearly, this uniformity is crucial for the sake of our inviscid limit result, specifically for satisfying condition \ref{Main2:H4} from \cref{Main2alt}. 

\begin{prop}\label{propestapriori}
Let $I\subset\mathbb{R}$ be an interval closed and bounded on the left with left endpoint $t_0$ and $\f\in L^2_{\rloc}(I,H)$ with $\nabla^\perp\cdot \f \in L^r_{\rloc}(I,L^r(\Omega))$, $2 \leq r < +\infty$. 	Then, for every weak solution $\bunu \in \mC_{\rloc}(I,\Wrw)$ of the 2D NSE \eqref{nse} on $I$ with forcing term $\f$  and for any $\nu_0 > 0$, the following inequalities hold for all $\nu > 0$ and $t \in I$:
  
\begin{align}\label{estapriori1}
 	|\bunu(t)|^2 \leq \left( |\bunu(t_0)|^2 + \frac{1}{\nu_0 \lambda_1} \|\f\|_{L^2(t_0,t;H)}^2  \right)e^{\nu_0 \lambda_1(t - t_0)},
 \end{align}

	\begin{align}\label{estapriori2}
	\| \nabla^\perp \cdot \bunu(t)\|_{L^r}^r \leq  \left( \|\nabla^\perp \cdot \bunu(t_0)\|_{L^r}^r + (\nu_0\lambda_1)^{1-r}\|\nabla^\perp \cdot \f\|_{L^r(t_0,t;L^r)}^r \right)e^{(r-1)\nu_0\lambda_1(t-t_0)},
	\end{align} 

 \begin{align}\label{estapriori3}
	&\|\partial_t \bunu\|_{L^2(t_0,t;V')} \leq c\lambda_1^{-1/2} \|\f\|_{L^2(t_0,t;H)}  \notag \\ 
	&\qquad + c\lambda_1^{-1/2+1/r} \Bigg[ \left( \nu +\left(|\bunu(t_0)|^2 + \frac{1}{\nu_0 \lambda_1} \|\f\|_{L^2(t_0,t;H)}^2  \right) e^{\nu_0 \lambda_1(t - t_0)} \right) \notag \\ 
 &\qquad\qquad\qquad\qquad \left( \|\nabla^\perp \cdot \bunu(t_0)\|_{L^r}^r + (\nu_0\lambda_1)^{1-r}\|\nabla^\perp \cdot \f\|_{L^r(t_0,t;L^r)}^r \right) \Bigg] e^{(r-1)\nu_0\lambda_1(t-t_0)} ,
	\end{align} 
	where $c>0$ is a universal constant.
\end{prop}

We present below the inviscid limit result for individual solutions that will be needed to verify some of the conditions from \cref{Main2alt}. The proof follows from standard arguments as in e.g. \cite[Chapter 4]{Lions2013}, but we include the details here for completeness.

\begin{prop}\label{thmNStoEuler}
	Let $I\subset\mathbb{R}$ be an interval closed and bounded on the left with left endpoint $t_0$, and let $\f\in L^2_{loc}(I, H)$ with $\nabla^\perp \cdot \f \in L^r_{loc}(I , L^r (\Omega) )$, $r\geq 2$. Let $\{\bunu\}_{\nu>0} \subset \mC_{\rloc}(I,\Wrw)$ be a vanishing viscosity net of weak solutions of the 2D Navier-Stokes equations \eqref{nse} on $I$ with external force $\f$, in the sense of \cref{NSweak}. Then, for every convergent subnet $\{\bu^{\nu'}\}_{\nu'}$ with $\bu^{\nu'} \to \bu$ in $\mC_{\rloc}(I,\Wrw)$ as $\nu' \to 0$, we have that the limit $\bu$ is a weak solution of the 2D Euler equations on $I$ with external force $\f$, in the sense of \cref{Eweak}.
\end{prop}
\begin{proof}
Suppose $\{\bu^{\nu'}\}_{\nu'}$ is a subnet converging to some $\bu$ in $\mathcal \Cloc(I,\Wrw)$ as $\nu' \to 0$. Let us show that $\bu$ is a weak solution of the 2D Euler equations on $I$.

Fix any compact subinterval $J \subset I$. Note that since, in particular, $\bu^{\nu'}(t_0) \to \bu(t_0)$ in $\Wrw$ as $\nu' \to 0$, then $\{\bu^{\nu'}(t_0)\}_{\nu'}$ is uniformly bounded in $\Wr$. Then, from the a priori bounds \eqref{estapriori2} and \eqref{estapriori3}, it follows that $\{\bu^{\nu'}\}_{\nu'}$ is uniformly bounded in $L^\infty(J,\Wr)$ and $\{\partial_t \bu^{\nu'}\}_{\nu'}$ is uniformly bounded in $L^2(J, V')$. Hence,  since $\Wr$ is compactly embedded in $H$ for $r>1$, we can apply Aubin-Lions Lemma (\cite[Theorem A.11]{FMRT2001}) to obtain that, up to a subnet, $\bu^{\nu'} \to \bu$ in $\mC(J,H) $ as $\nu' \to 0$. In particular, $\bu \in L^\infty(J,H)$ and, consequently, $\bu \in  L^\infty_\rloc(I,H)$. Moreover, since $\bu \in \mC_\rloc (I, \Wrw) \subset \mC_{\rloc}(I,\Tsp')$, we deduce that condition \ref{Eweaki} of \cref{Eweak} is satisfied.

To verify the remaining condition, \ref{Eweakii}, fix any test function $\varphi \in \mC^\infty_\rc(I)$ and $\bv \in \Tsp$. By assumption, we have that
\begin{multline}\label{NSweakF}
	- \int_I \int_\Omega (\bu^{\nu'} \cdot \bv) \varphi' \;\rd \bx \;\rd t + \nu' \int_I \int_\Omega ( \nabla \bu^{\nu'} : \nabla \bv) \varphi \;\rd \bx \; \rd t - \int_I \int_\Omega ( \bu^{\nu'} \otimes \bu^{\nu'} : \nabla \bv) \varphi \;\rd \bx \;\rd t \\
	= \int_\Omega (\bu^{\nu'}(t_0) \cdot \bv) \varphi(t_0) \; \rd \bx + \int_I \langle \f, \bv\rangle_{V',V} \varphi \;\rd t
\end{multline}
for every $\nu'$.

Since $\varphi$ has compact support in $I$, in view of the convergence $\bu^{\nu'} \to \bu$ in $\mathcal \Cloc(I,\Wrw)$, we immediately obtain
\begin{gather*}
	\int_I \int_\Omega (\bu^{\nu'} \cdot \bv) \varphi' \;\rd \bx \;\rd t \longrightarrow 	\int_I \int_\Omega (\bu \cdot \bv) \varphi' \;\rd \bx \;\rd t, \quad \mbox{ as } \nu' \to 0, \\
	\nu' \int_I \int_\Omega ( \nabla \bu^{\nu'} : \nabla \bv) \varphi \;\rd \bx \; \rd t = 
	- \nu' \int_I \int_\Omega ( \bu^{\nu'} : \Delta \bv) \varphi \;\rd \bx \; \rd t \longrightarrow 0, \quad \mbox{ as } \nu' \to 0, \\
	\int_\Omega (\bu^{\nu'}(t_0) \cdot \bv) \varphi(t_0) \; \rd t  \longrightarrow  \int_\Omega (\bu (t_0) \cdot \bv) \varphi(t_0) \; \rd t, \quad \mbox{ as } \nu' \to 0.
\end{gather*}

Regarding the nonlinear term in \eqref{NSweakF}, we proceed as follows. Let $J \subset I$ be a compact subinterval containing the support of $\varphi$. Note that
\begin{align*}
	&\left| \int_I \int_\Omega (\bu^{\nu'} \otimes \bu^{\nu'} : \nabla \bv) \varphi \; \rd \bx \; \rd t  -   \int_I \int_\Omega (\bu \otimes \bu : \nabla \bv) \varphi \; \rd \bx \; \rd t \right| \\
	&\qquad\qquad= \left| \int_J \int_\Omega \left[ \left( (\bu^{\nu'} - \bu) \otimes \bu^{\nu'} + \bu \otimes (\bu^{\nu'} - \bu) \right) : \nabla \bv \right] \varphi \; \rd \bx \; \rd t  \right| \\
	&\qquad\qquad\leq  \|\bu^{\nu'}-\bu\|_{L^\infty(J;L^2(\Omega))}(\|\bu^{\nu'}\|_{L^\infty(J;L^2(\Omega))}+
	\|\bu\|_{L^\infty(J;L^2(\Omega))}) \|\nabla \bv \|_{L^\infty(\Omega)} \|\varphi\|_{L^1(J)}.
\end{align*}
Since $\bu^{\nu'} \to \bu$ in $L^\infty(J,H)$ as $\nu' \to 0$, it follows that 
\begin{align*}
	\int_I \int_\Omega (\bu^{\nu'} \otimes \bu^{\nu'} : \nabla \bv) \varphi \; \rd \bx \; \rd t  \longrightarrow \int_I \int_\Omega (\bu \otimes \bu : \nabla \bv) \varphi \; \rd \bx \; \rd t, \quad \mbox{ as } \nu' \to 0.
\end{align*}
Therefore, passing to the limit as $\nu' \to 0$ in \eqref{NSweakF}, we deduce that $\bu$ satisfies item \ref{Eweakii} of \cref{Eweak}. This concludes the proof.
\end{proof}

\subsubsection{3D incompressible flows}\label{subsubsec:3D}

Let us again take $I \subset \RR$ to be an interval closed and bounded on the left with left endpoint $t_0$. We recall the following standard notion of weak solution to the 3D Navier-Stokes equations.

\begin{definition}\label{3NSweak}Let $\f\in L^2_{\rloc}(I,V')$. We say that $\bu$ is a \textbf{Leray-Hopf weak solution of the 3D Navier-Stokes equations \eqref{nse} on $I$} if
\begin{enumerate}[label={(\roman*)}]
 \item $\bu\in  \mC_{\rloc}(I,H_\rw) \cap L_{\rloc}^2(I, V) $; 

 \item $\partial_t\bu \in   L_{\rloc}^{4/3}(I, V')$;

 \item \label{3NSweakeq} For every function $\Phi\in \mC^\infty(I\times\mathbb{R}^3)^3$ that is $\Omega$-periodic, divergence-free and compactly supported on $I$, it holds 
\begin{align}\label{eq3b3} 
\int_I\int_\Omega \bu \cdot\partial_t\Phi \;\rd\bx \;\rd t - \nu \int_I \int_\Omega \nabla \bu : \nabla \Phi \;\rd \bx \;\rd t + 
\int_I\int_\Omega \bu\otimes \bu: \nabla \Phi \;\rd \bx \;\rd t \notag \\
=-\int_\Omega \bu(t_0)\cdot\Phi(t_0)\;\rd \bx -\int_I\int_\Omega \f\cdot\Phi\;\rd \bx \;\rd t.
\end{align}

 \item\label{energy:ineq:3D:nse} $\bu$ satisfies the following energy inequality for almost
all $t'\in I$ and for all $t\in I$ with $t>t'$:
\begin{equation}\label{energy-ineq}
\frac{1}{2}
|\bu(t)|^2+\nu\int_{t'}^{t}\|\bu(s)\|^2\;\rd s \leq \frac{1}{2}
|\bu(t')|^2+\int_{t'}^t\langle\f(s),\bu(s)\rangle_{V',V}\;\rd s.
\end{equation}

 \item\label{3nsedefweaksolv} If $I$ is closed and bounded on the left, with left endpoint
$t_0$, then $\bu$ is strongly continuous in $H$ at $t_0$ from the right,
i.e., $\bu(t)\rightarrow \bu(t_0)$ in $H$ as $t \rightarrow t_0^+$.
\end{enumerate}
\end{definition}

The set of allowed times $t'$ in \eqref{energy-ineq} are characterized as the points of strong continuity from the right of $\bu$ in $H$. In particular, condition \ref{3nsedefweaksolv} implies that $t'=t_0$ is allowed in that case.

We note that condition \ref{energy:ineq:3D:nse} can be interchanged with the following inequality in the sense of distributions on $I$:
\begin{align}\label{energy:ineq:2}
	\frac{1}{2} \frac{\rd}{\rd t} |\bu(t)|^2 + \nu \|\bu(t)\|^2 \leq \langle \f(t), \bu(t)\rangle_{V',V},
\end{align}
see e.g. \cite{FRT2013}.

Given any $\f \in L^2_{\rloc}(I,V')$ and initial datum $\bu_0 \in H$, it is well known that there exists at least one Leray-Hopf weak solution of the 3D Navier-Stokes equations, \eqref{nse}, defined on $I$ and satisfying $\bu(t_0) = \bu_0$. For a proof of this classical result, we refer to e.g. \cite{bookcf1988,Lady,Lions1969,temam84,temam1995}.

Regarding the 3D Euler equations, we consider the notion of dissipative solution introduced in \cite[Section 4.4]{Lions2013}, where the forcing term was taken to be zero for simplicity. With the appropriate modifications to include an external force, we obtain the following definition. 

\begin{definition}\label{Edissipative} Let $\f \in L^1_{\rloc}(I,H)$.  We say that $\bu$ is a \textbf{dissipative solution of the 3D Euler equations \eqref{ee} on $I$} if
\begin{enumerate}[label={(\roman*)}]
 \item\label{Edissi} $\bu\in  \mC_{\rloc}(I, H_\rw)$; 
 \item \label{Edissii} For every $\bv\in \mC_{\rloc}(I,H)$ such that $d(\bv)=\frac{1}{2}(\nabla \bv+(\nabla \bv)^T)\in L^1_{\rloc}(I,L^\infty)$ and $E(\bv)=-\partial_t\bv-\mathbb P[(\bv\cdot \nabla)\bv)]\in L^1_{\rloc}(I,H)$, where $\bP$ denotes the projection onto $\Omega$-periodic divergence-free vector fields with zero spatial average, it holds 
\begin{multline}\label{3dEweakiiineq}
\int_\Omega |\bu(t) -\bv(t)|^2\;\rd\bx \leq \exp\left(2\int_{t_0}^t\|d^-(\bv)\|_{L^\infty} \;\rd s\right)\int_\Omega |\bu(t_0)-\bv(t_0)|^2\;\rd \bx \\
+2\int_{t_0}^t\int_\Omega \exp\left(2\int_{s}^t\|d^-(\bv)\|_{L^\infty} \;\rd \tau\right)(E(\bv)+\f)\cdot(\bu-\bv)\;\rd \bx \;\rd s,
\end{multline}
for all $t \in I$, where $d^-(\bv)=(\inf\{\xi^T d(\bv)\xi: \xi \in \RR^2, |\xi|=1\})^-$ is the negative part of the smallest eigenvalue of $d(\bv)$.
\end{enumerate}
\end{definition}

As mentioned in \cref{subsec:setting}, the main motivation behind this definition comes from establishing a notion of solution to the 3D Euler equations that is obtained as an appropriate limit of a vanishing viscosity sequence of Leray-Hopf weak solutions of the 3D NSE. This is indeed how the existence of a dissipative solution is shown in  \cite[Proposition 4.2]{Lions2013}, under an initial condition $\bu(t_0) = \bu_0$, for any $\bu_0 \in H$, and in the absence of external forcing term. With a simple adaptation, one can show the same holds for any given forcing $\f \in L^1_{\rloc}(I,H)$.

Following analogous steps from this proof, we obtain the inviscid limit result in \cref{thm3dNStoEuler} below, which we later apply for verifying  condition \ref{ii:Main1} of \cref{Main1}. We present the details of its proof here for completeness. Before proceeding, we show in the following proposition a few useful a priori estimates regarding weak solutions of the 3D NSE. For this formulation and the subsequent results, we require the forcing term to be in $L^2_{\rloc}(I,H),$ so it fits both \cref{3NSweak,Edissipative}.

\begin{prop}\label{prop:apriori:est:3D}
	Let $I\subset\mathbb{R}$ be an interval closed and bounded on the left with left endpoint $t_0$ and $\f\in L^2_{\rloc}(I,H)$. Let $\nu_0 >0$. Then, for every Leray-Hopf weak solution $\bunu$ of the 3D NSE \eqref{nse} on $I$ with forcing term $\f$ and for all $\nu > 0$, the following inequalities hold:
	\begin{align}\label{apriori:est:1}
		|\bunu(t)|^2 + 2 \nu \int_{t_0}^t e^{\nu_0 \lambda_1(t - \tau)} \|\bunu(\tau)\|^2 \rd \tau \leq e^{\nu_0 \lambda_1(t - t_0)} \left[ |\bunu(t_0)|^2 + \frac{1}{\nu_0 \lambda_1} \|\f\|_{L^2(t_0,t;H)}^2  \right],
	\end{align}
	for all $t \in I$, and
	\begin{multline}\label{apriori:est:2}
	 	\| \bunu(t) - \bunu(s)\|_{(\Winf)'} \\
	 	\leq c |t-s|^{1/2} (\nu^{1/2} + \nu_0^{1/2}) \lambda_1^{-3/4}  e^{\frac{\nu_0 \lambda_1(t - t_0)}{2}}  \left[ |\bunu(t_0)|^2 + \frac{1}{\nu_0 \lambda_1}  \|\f\|_{L^2(t_0,t;H)}^2  \right]^{1/2} \\
	 	+ |t - s| e^{\nu_0 \lambda_1(t - t_0)}  \left[ |\bunu(t_0)|^2 + \frac{1}{\nu_0 \lambda_1}  \|\f\|_{L^2(t_0,t;H)}^2  \right] ,
	\end{multline}
	for all $s,t \in I$ with $s \leq t$, and for some positive constant $c$ which is independent of $\nu$.	
\end{prop}

\begin{proof}
From \eqref{energy:ineq:2}, it follows that, for every non-negative test function $\varphi \in \mC^\infty_\rc(I)$,
\begin{align*}
	- \frac{1}{2} \int_{t_0}^t |\bunu(\tau)|^2 \varphi'(\tau) \rd \tau + \nu \int_{t_0}^t \| \bunu(\tau)\|^2 \varphi(\tau) \rd \tau \leq \int_{t_0}^t (\f(\tau), \bunu(\tau)) \varphi(\tau) \rd \tau,
\end{align*}	
for all $t \in I$. Then, by choosing an appropriate sequence of test functions on $I$ and invoking the Lebesgue differentiation theorem, together with the fact that $\bunu \in \mC_{\rloc}(I,H_\rw)$ and $\bunu$ is strongly continuous at $t_0$ from the right, we deduce that
\begin{align}
	&|\bunu(t)|^2 \psi(t) + 2 \nu \int_{t_0}^t \|\bunu(\tau)\|^2 \psi(\tau) \rd \tau  \notag \\
	&\qquad \leq |\bunu(t_0)|^2 \psi(t_0) + \int_{t_0}^t |\bunu(\tau)|^2 \psi'(\tau) \rd \tau + 2 \int_{t_0}^t (\f(\tau), \bunu(\tau)) \psi(\tau) \rd \tau,
	\label{energy:ineq:psi}
\end{align}
for all $t \in I$ and for every non-negative function $\psi \in \mC^1(I)$; see e.g. \cite[Chapter II, Appendix B.1]{FMRT2001} for a similar argument. 

In particular, choosing $\psi(t) = e^{-\nu_0 \lambda_1 t}$, and estimating the integrand in the last term of \eqref{energy:ineq:psi} as 
\begin{align*}
	(\f(\tau), \bunu(\tau)) \leq |\f(\tau)| |\bunu(\tau)| 
	\leq \frac{1}{2\nu_0 \lambda_1} |\f(\tau)|^2 + \frac{\nu_0 \lambda_1}{2} |\bunu(\tau)|^2, 
\end{align*}
it follows that 
\begin{align*}
	|\bunu(t)|^2 e^{-\nu_0 \lambda_1 t} + 2 \nu \int_{t_0}^t \|\bunu(\tau)\|^2 e^{-\nu_0 \lambda_1 \tau} \rd \tau 
	&\leq |\bunu(t_0)|^2 e^{-\nu_0 \lambda_1 t_0} + \frac{1}{\nu_0 \lambda_1} \int_{t_0}^t |\f(\tau)|^2 e^{-\nu_0 \lambda_1 \tau} \rd \tau \\
	&\leq |\bunu(t_0)|^2 e^{-\nu_0 \lambda_1 t_0} + \frac{e^{-\nu_0 \lambda_1 t_0}}{\nu_0 \lambda_1}  \|\f\|_{L^2(t_0,t;H)}^2,
\end{align*}
which immediately yields \eqref{apriori:est:1}.

Regarding \eqref{apriori:est:2}, let $s, t \in I$ with $s \leq t$. Since $\partial_t \bunu \in L^{4/3}_{\rloc}(I,V') \subset L^1_{\rloc}(I, (\Winf)'),$ then, for all $\bv \in \Winf,$
\begin{multline*}
	|\langle \bunu (t) - \bunu (s), \bv \rangle_{(W_{\sigma}^{1,\infty})',W_{\sigma}^{1,\infty}}| 
	= \left| {\left\langle \int_s^{t} \partial_t \bunu (\tau) \rd \tau , \bv \right\rangle}_{(W_{\sigma}^{1,\infty})',W_{\sigma}^{1,\infty}} \right| \\
	\leq \left( \int_s^{t} \|\partial_t \bunu (\tau)\|_{(W_{\sigma}^{1,\infty})'} \rd \tau \right) \|\bv\|_{W_{\sigma}^{1,\infty}}
	\leq |t-s|^{1/2} \| \partial_t \bunu\|_{L^2(s,t;(W_{\sigma}^{1,\infty})')} \|\bv\|_{W_{\sigma}^{1,\infty}}.
\end{multline*}
Hence,
\begin{align*}
	\|  \bunu (t) - \bunu (s) \|_{(W_{\sigma}^{1,\infty})'}
	\leq |t-s|^{1/2} \| \partial_t \bunu\|_{L^2(s,t;(W_{\sigma}^{1,\infty})')}.
\end{align*}

We proceed to obtain an estimate of $\| \partial_t \bunu\|_{L^2(s,t;(W_{\sigma}^{1,\infty})')}$. From \eqref{eq3b3}, it follows that, for all $\bv \in V,$
\begin{align}\label{eq:dist}
	\frac{\rd}{\rd t} (\bunu, \bv) + \nu (\nabla \bunu, \nabla \bv) - (\bunu \otimes \bunu, \nabla \bv) = (\f, \bv)
\end{align}
in the sense of distributions on $I$. In particular, again since $\partial_t \bunu \in L^1_{\rloc}(I, (\Winf)')$ then for all $\bv \in \Winf$
\begin{align*}
	{\left\langle \partial_t \bunu, \bv \right\rangle}_{(\Winf)', \Winf}
	= \frac{\rd}{\rd t} \langle \bunu, \bv \rangle_{(\Winf)', \Winf} 
	= \frac{\rd}{\rd t} (\bunu, \bv).
\end{align*}
From \eqref{eq:dist}, along with Cauchy-Schwarz, H\"{o}lder's inequality, Poincar\'e inequality \eqref{ineq:Poincare}, and the inequality \eqref{ineq:Wr:V} for $r=\infty,$ we thus obtain 
\begin{align*}
	{\left\langle \partial_t \bunu (\tau), \bv \right\rangle}_{(\Winf)', \Winf}
	&\leq \nu \|\bunu(\tau)\| \| \bv\| + |\bunu(\tau)|^2 \|\bv\|_{\Winf} + |\f(\tau)| |\bv| \\
	&\leq c\nu \lambda_1^{-3/4} \|\bunu(\tau)\| \|\bv\|_{\Winf} + |\bunu(\tau)|^2 \|\bv\|_{\Winf} + c \lambda_1^{-5/4} |\f(\tau)| \|\bv\|_{\Winf},
\end{align*}
for a.e. $\tau \in I$.
Consequently,
\begin{align*}
	\|\partial_t \bunu (\tau)\|_{(\Winf)'}
	\leq c \lambda_1^{-3/4} ( \nu \|\bunu (\tau)\| + \lambda_1^{-1/2} |\f(\tau)| ) + |\bunu(\tau)|^2,  
	 \quad \mbox{for a.e. } \tau \in I.
\end{align*}

Hence, 
\begin{multline}\label{int:Dtu:Winf}
	\| \partial_t \bunu \|_{L^2(s,t;(\Winf)')} 
	\leq c \lambda_1^{-3/4} \left( \nu \|\bunu\|_{L^2(s,t;V)} + \lambda_1^{-1/2}  \|\f\|_{L^2(s,t;H)} \right) + \|\bunu\|^2_{L^{4}(s,t;H)} \\
	\leq  c \lambda_1^{-3/4}  \left( \nu \|\bunu\|_{L^2(t_0,t;V)} + \lambda_1^{-1/2}  \|\f\|_{L^2(t_0,t;H)} \right) +  |t-s|^{1/2} \|\bunu\|_{L^\infty(t_0,t;H)}^{2}.
\end{multline}

Thus, \eqref{apriori:est:2} follows by invoking \eqref{apriori:est:1} to further estimate the right-hand side of \eqref{int:Dtu:Winf}.
\end{proof}

\begin{prop}\label{thm3dNStoEuler}
Let $I\subset\mathbb{R}$ be an interval closed and bounded on the left with left endpoint $t_0$, and let $\f \in L^2_{\rloc}(I,H)$. Consider also a (strongly) compact set $K$ in $H$, and let $\{\bunu\}_{\nu>0}$ be a vanishing viscosity net of Leray-Hopf weak solutions to the 3D Navier-Stokes equations \eqref{nse} on $I$ with external force $\f$ and with initial data $\bunu(t_0)\in K$, in the sense of \cref{3NSweak}. Then, for every convergent
subnet $\{\bu^{\nu'}\}_{\nu'}$ with $\bu^{\nu'} \to \bu$ in $\mC_{\rloc}(I,H_{\rw})$ as $\nu' \to 0$, we have that the limit $\bu$ is a dissipative solution of the 3D Euler equations on $I$ with external force $\f$, in the sense of \cref{Edissipative}.
\end{prop}
\begin{proof}
Let $\{\bu^{\nu'}\}_{\nu'}$ be a subnet converging to some $\bu$ in $\mathcal \Cloc(I,H_\rw)$ as $\nu' \to 0$. Thus, $\bu$ satisfies condition \ref{Edissi} of \cref{Edissipative}. 

Now let us prove that $\bu$ satisfies condition \ref{Edissii}. By a simple density argument, it suffices to show that \eqref{3dEweakiiineq} holds for any test function $\bv \in \mC^\infty(I \times \mathbb{R}^3)^3$ that is $\Omega$-periodic, divergence-free and compactly supported on $I$ (see \cite[Section 4.4]{Lions2013}). Let $\bv$ be such a test function. From \eqref{eq3b3}, it follows that, for every $\nu',$
\begin{align*}
	\frac{\rd}{\rd t} (\bunup, \bv) - (\bunup, \partial_t \bv) + \nu' (\nabla \bunup, \nabla \bv) - (\bunup \otimes \bunup, \nabla \bv) = (\f, \bv)
\end{align*}
in the sense of distributions on $I$.

Since $-\partial_t\bv= E(\bv)+\mathbb P[(\bv\cdot\nabla) \bv]$ and
\begin{multline*}
	(\bunup, \bP [(\bv \cdot \nabla)\bv]) - (\bunup \otimes \bunup, \nabla \bv)
	= (\bunup, (\bv \cdot \nabla)\bv) - ((\bunup \cdot \nabla)\bv, \bunup) \\
	= - ([(\bunup - \bv) \cdot \nabla] \bv, \bunup)
	= - ([(\bunup - \bv) \cdot \nabla] \bv, \bunup - \bv),
\end{multline*}
then
\begin{align*}
\frac{d}{dt} (\bunup, \bv) + (\bunup, E(\bv)) - ([(\bunup - \bv) \cdot \nabla] \bv, \bunup - \bv) + \nu' (\nabla \bunup, \nabla \bv) = (\f, \bv).
\end{align*}

Note also that
\begin{align*}
	\frac{d}{dt} |\bv|^2 = 2 (\bv, \partial_t \bv) = - 2 (\bv, E(\bv)).
\end{align*}
Combining the last two equations with the energy inequality $\frac{d}{dt} |\bunup|^2 \leq 2 (\f, \bunup)$, which follows from \eqref{energy:ineq:2}, we obtain  
\begin{multline*}
\frac{d}{dt} |\bunup - \bv|^2 \leq 2 (\bunup - \bv, E(\bv) + \f) - 2 ([(\bunup - \bv) \cdot \nabla] \bv, \bunup - \bv) + 2\nu' (\nabla \bunup, \nabla \bv).
\end{multline*}

Observe that 
\begin{align*}
	&- 2 ([(\bunup - \bv) \cdot \nabla] \bv, \bunup - \bv) = - 2 \int_\Omega [((\bunup-\bv)\cdot \nabla)\bv]\cdot (\bunup-\bv) \rd \bx \\
	&\qquad\qquad = -2 \int_\Omega (\bunup-\bv)_i\partial_i\bv_j(\bunup-\bv)_j \,\rd \bx \\
	&\qquad\qquad= -2 \int_\Omega (\bu^{\nu'}-\bv)_i\frac{\partial_i\bv_j+\partial_j\bv_i}{2}(\bu^{\nu'}-\bv)_j \,\rd \bx
	= -2 \int_\Omega(d(\bv)(\bu^{\nu'}-\bv))\cdot(\bu^{\nu'}-\bv) \,\rd \bx \\
	&\qquad\qquad \leq 2 \int_\Omega \sup\{-(d(\bv)\xi)\cdot \xi: |\xi|=1\}|\bu^{\nu'}-\bv|^2 \, \rd \bx \\
	&\qquad\qquad = -2 \int_\Omega \inf\{\xi^T d(\bv)\xi: |\xi|=1\}|\bu^{\nu'}-\bv|^2 \, \rd \bx \\
	&\qquad\qquad \leq 2 \int_\Omega d^-(\bv)|\bu^{\nu'}-\bv|^2 \, \rd \bx
	\leq 2\| d^-(\bv)\|_{L^\infty}|\bu^{\nu'}-\bv|^2.
\end{align*}

Thus,
\begin{align*}
	\frac{\rd}{\rd t} |\bunup - \bv|^2 \leq 2 (\bunup - \bv, E(\bv) + \f) + 2\| d^-(\bv)\|_{L^\infty}|\bu^{\nu'}-\bv|^2 + 2 \nu' \| \bv\| \|\bunup\|
\end{align*}
in the sense of distributions on $I$. Choosing an appropriate sequence of test functions on $I$ and invoking the Lebesgue differentiation theorem, similarly as in  \cite[Chapter II, Appendix B.1]{FMRT2001}, we obtain the following Gronwall-type inequality
\begin{align}\label{ineq:u:v}
	|\bunup(t) - \bv(t)|^2 
	&\leq \exp \left( 2 \int_{t_0}^t \| d^-(\bv)\|_{L^\infty} \, \rd s \right) |\bunup(t_0) - \bv(t_0)|^2 \notag\\
	&\qquad + 2 \int_{t_0}^t  \exp \left( 2 \int_s^t \| d^-(\bv)\|_{L^\infty} \, \rd \tau \right)  (\bunup - \bv, E(\bv) + \f) \, \rd s \notag\\
	&\qquad + 2 \nu' \int_{t_0}^t \exp \left( 2 \int_s^t \| d^-(\bv)\|_{L^\infty} \, \rd \tau \right) \|\bv\| \|\bunup\| \,\rd s,
\end{align}
for all $t \in I$.

Fix $\nu_0 > 0$ such that all parameters $\nu'$ satisfy $\nu' \leq \nu_0$. From \eqref{apriori:est:1}, we can bound the last term in the right-hand side of \eqref{ineq:u:v} by
\begin{multline}\label{ineq:last:term}
	2 \nu'^{1/2} \exp \left( 2 \int_{t_0}^t \| d^-(\bv)\|_{L^\infty} \, \rd \tau  \right) \|\bv\|_{L^2(t_0,t;V)} \left( \nu' \int_{t_0}^t \|\bunup\|^2 \, \rd s \right)^{1/2} \\
	\leq c \nu'^{1/2} \exp \left( 2 \int_{t_0}^t \| d^-(\bv)\|_{L^\infty} \, \rd \tau  \right) \|\bv\|_{L^2(t_0,t;V)} e^{\frac{\nu_0 \lambda_1}{2} (t - t_0)} \left[ |\bunup(t_0)|^2 + \frac{1}{\nu_0 \lambda_1}\|\f\|_{L^2(t_0,t;H)}^2  \right]^{1/2}.
\end{multline}
Since the net $\{\bunup(t_0)\}_{\nu'}$ is in the compact set $K$, and hence is bounded in $H$, then the expression in the right-hand side of \eqref{ineq:last:term} vanishes as $\nu' \to 0$. 

Moreover, since $\bunup \to \bu$ in $\mC_{\rloc}(I,H_{\rw})$, then together with the bound \eqref{apriori:est:1} it follows that the second term in the right-hand side of \eqref{ineq:u:v} converges as $\nu' \to 0$ to 
\begin{align*}
	2 \int_{t_0}^t  \exp \left( 2 \int_s^t \| d^-(\bv)\|_{L^\infty} \, \rd \tau \right)  (\bu - \bv, E(\bv) + \f) \, \rd s .
\end{align*}

Additionally, again since $\{\bunup(t_0)\}_{\nu'} \subset K$ then, modulo a subnet, we have the strong convergence: $\bunup(t_0) \to \bu(t_0)$ in $H$. Combining these facts, we obtain by taking the $\liminf$ as $\nu' \to 0$ in \eqref{ineq:u:v} that
\begin{multline*}
	|\bu(t) - \bv(t)|^2 \leq \liminf_{\nu' \to 0} 	|\bunup(t) - \bv(t)|^2
	\leq \exp \left( 2 \int_{t_0}^t \| d^-(\bv)\|_{L^\infty} \, \rd s \right) |\bu(t_0) - \bv(t_0)|^2 \\
	+ 2 \int_{t_0}^t  \exp \left( 2 \int_s^t \| d^-(\bv)\|_{L^\infty} \, \rd \tau \right)  (\bu - \bv, E(\bv) + \f) \, \rd s,
\end{multline*}
which shows \eqref{3dEweakiiineq}, and concludes the proof.
\end{proof}

\subsection{Convergence of statistical solutions of 2D Navier-Stokes to 2D Euler} 
\label{subsec:2D:NSE:Euler}

In this section, we verify the assumptions of \cref{Main2alt} to deduce the convergence of a net of trajectory statistical solutions of the 2D NSE towards a trajectory statistical solution of the 2D Euler equations, as stated in \cref{MainThmApp2D} below.

We start by fixing the required setting from \cref{Main2alt}. Let $I \subset \RR$ be an interval closed and bounded on the left with left endpoint $t_0$. Take $X = \Wrw$, for any given $1 < r < \infty$, and define, for each fixed $\nu > 0$,
\begin{eqnarray}\label{def:Snu}
S_\nu: \Wrw &\rightarrow & \mC_{\rloc}(I,\Wrw) \nonumber \\
\bu_0 &\mapsto& \bunu,
\end{eqnarray}
where $\bunu$ is the unique weak solution of \eqref{nse} on $I$ in the sense of \cref{NSweak} satisfying $\bunu(t_0) = \bu_0$. Thus, the operator $P_\nu=\Pi_{t_0}S_\nu$, as defined in \cref{Main2alt}, is the identity operator.

Note that since $\Wr$ is a separable Banach space then every Borel probability measure on $\Wrw$ is also a Borel probability measure on $\Wr$ (and vice-versa), and hence tight in $\Wr$, i.e. inner regular with respect to the family of compact subsets of $\Wr$ (see \cref{subsec:meas:theory}), which are also compact sets in $\Wrw$. In summary, every Borel probability measure on $\Wrw$ (or, equivalently, $\Wr$) is tight in $\Wrw$. For this reason, we consider $\mu_0$ as any Borel probability measure in $\Wr$ in the statements of this section. 

Then, given a Borel probability measure $\mu_0$ on $\Wr,$ we set, for simplicity, $\mu_\nu=\mu_0$ for all $\nu>0$.  Thus, $P_\nu\mu_\nu= \mu_\nu = \mu_0$ and assumption \ref{Main2:H0} is immediately satisfied. Also, from the tightness of $\mu_0$ we obtain the existence of a sequence $\{K_n\}_{n\in\NN}$ of compact sets in $\Wrw$ satisfying \ref{Main2:H2}.

As we shall see, assumptions \ref{Main2:H1}, \ref{Main2:H3} and \ref{Main2:H4} actually hold for any compact set $K$ of $\Wrw$, with $\mU$ defined as
\begin{eqnarray}\label{ue2d}
	\mU_I = \left\{\begin{array}{cc}\bu \in \Cloc(I,\Wrw):& \bu \mbox{ is a weak solution of the 2D Euler equations}\\
	&\mbox{\eqref{ee} on  $I$ in the sense of \cref{Eweak}}
	\end{array}\right\}.
\end{eqnarray}

For simplicity, we assume throughout this section that $r$ is restricted to the range $2 \leq r < \infty$. In particular, this allows us to obtain the bound \eqref{estapriori2} below for the $L^r$-norm of the vorticity associated with a weak solution of the 2D NSE. We note, however, that the case $1 < r < 2$ can also be treated, by appealing to the notion of renormalized solutions, see \cite[Section 4.1]{Lions2013}. This case is indeed considered in the work \cite{WW2022}, where an analogous convergence result for trajectory statistical solutions of the 2D NSE towards a trajectory statistical solution of 2D Euler is obtained, albeit under the assumption of zero forcing term and with a slightly different setting than ours, particularly concerning the definition of $X$ and the fact that the spatial domain is taken as $\RR^2$.

The next proposition shows that condition \ref{Main2:H1} from \cref{Main2alt} holds true in this context.

\begin{prop}\label{contSnu}
	Let $K$ be a compact set in $\Wrw$, $2\leq r<\infty$. Then, for each $\nu > 0$, the operator $S_{\nu}|_K: K \rightarrow \mC_{\rloc}(I,\Wrw)$ is continuous.
\end{prop}
\begin{proof}
Since the weak topology is metrizable on bounded subsets of $\Wr$, it suffices to show that, for any given $\bu_0 \in K$ and any sequence $\{\bu_{0,n}\}_{n \in \NN}$ in $K$ converging weakly to $\bu_0,$ it follows that $S_\nu(\bu_{0,n})$ converges to $S_\nu(\bu_0)$ in $\Cloc(I,\Wrw)$.

Consider any compact subinterval $J \subset I$ with left endpoint $t_0$. It is sufficient to show that $S_\nu(\bu_{0,n})$ converges to $S_\nu(\bu_0)$ in $\mC(J,\Wrw)$. We first show that $\{S_\nu(\bu_{0,n})\}_n$ is relatively compact in $\mC(J,\Wrw)$. 

Since $\{\bu_{0,n}\}_n$ is contained in the compact set $K$, then $\{\bu_{0,n}\}_n$ is a bounded sequence in $\Wr$. Thus, from \eqref{divcurl} and \eqref{estapriori2} it follows that $\{S_\nu(\bu_{0,n})\}_n$ is uniformly bounded in $\mC(J, \Wr)$. We may thus consider a ball $B_{\Wr}(R)$ in $\Wr$, $R > 0$, such that $S_\nu(\bu_{0,n})(t) \subset B_{\Wr}(R)$ for all $n \in \NN$ and $t \in J$. Note also that, from \eqref{estapriori3}, it follows that $\{\partial_t S_\nu(\bu_{0,n})\}_n$ is uniformly bounded in $L^2(J,V')$.

Let $\bw \in \Wrp$ and $\varepsilon > 0$. Since $V$ is dense in $\Wrp$, we may take $\bv \in V$ such that $\| \bv - \bw \|_{\Wrp} < \varepsilon/(4R)$. Hence, for all $n \in \NN$ and $s < t$, we have
\begin{multline}\label{ineq:Snuu0n}
	|\langle \bw, S_\nu(\bu_{0,n})(t) - S_\nu(\bu_{0,n})(s) \rangle_{\Wrp, \Wr}| \\
	\qquad\leq |\langle \bw - \bv, S_\nu(\bu_{0,n})(t) - S_\nu(\bu_{0,n})(s) \rangle_{\Wrp,\Wr}|  + |\langle S_\nu(\bu_{0,n})(t) - S_\nu(\bu_{0,n})(s), \bv \rangle_{V',V}| .
\end{multline}
The first term is estimated as 
\begin{align*}
	|\langle \bw - \bv, S_\nu(\bu_{0,n})(t) - S_\nu(\bu_{0,n})(s) \rangle_{\Wrp,\Wr}|  
	\leq 2R \|\bw - \bv\|_{\Wrp} < \frac{\varepsilon}{2}.
\end{align*}
Regarding the second term in \eqref{ineq:Snuu0n}, we have
\begin{multline}\label{ineq:Snuu0n:2}
	|\langle S_\nu(\bu_{0,n})(t) - S_\nu(\bu_{0,n})(s), \bv \rangle_{V',V}| 
	= \left| {\left\langle \int_s^t \partial_t S_\nu(\bu_{0,n})(\tau) \rd \tau, \bv \right\rangle}_{V',V} \right| \\
	\leq \left( \int_s^t \|\partial_t S_\nu(\bu_{0,n})(\tau)\|_{V'} \rd \tau \right) \|\bv \|
	\leq |t-s|^{1/2} \| \partial_t S_\nu(\bu_{0,n})\|_{L^2(J,V')} \|\bv\| \\
	\leq C |t-s|^{1/2} \|\bv\|.
\end{multline}
Hence, it follows from \eqref{ineq:Snuu0n}-\eqref{ineq:Snuu0n:2} that, for all $s,t \in J$ with $|t-s| < \varepsilon^2/(2C \|\bv\|)^2,$
\begin{align*}
	|\langle \bw, S_\nu(\bu_{0,n})(t) - S_\nu(\bu_{0,n})(s) \rangle_{\Wrp, \Wr}| 
	< \frac{\varepsilon}{2} + C |t-s|^{1/2} \|\bv\| < \varepsilon.
\end{align*}
Since $\bw$ and $\varepsilon$ are arbitrary, this implies that $\{S_\nu(\bu_{0,n})\}_n$ is equicontinuous in $\mC(J, B_{\Wr}(R)_\rw)$. Moreover, since $B_{\Wr}(R)$ is a compact set in $\Wrw$, we also have that, for each fixed $t \in J$, $\{S_\nu(\bu_{0,n})(t)\}_n$ is relatively compact in $B_{\Wr}(R)_\rw$. Therefore, by the Arzel\`a-Ascoli theorem, it follows that $\{S_\nu(\bu_{0,n})\}_n$ is relatively compact in $\mC(J,B_{\Wr}(R)_\rw)$, and hence in $\mC(J,\Wrw)$.

Thus, there exists a subsequence $\{S_\nu(\bu_{0,n'})\}_{n'}$ and $\tilde\bu \in \mC(J,\Wrw)$ such that $S_\nu(\bu_{0,n'}) \to \tilde\bu$ in $\mC(J,\Wrw)$. In particular, $S_\nu(\bu_{0,n'})(t_0)=\bu_{0,n'}$ converges weakly to $\tilde\bu(t_0)$ in $\Wr$ and, by uniqueness of the limit, $\tilde\bu(t_0) = \bu_0$. Also, by \cref{thmNStoEuler}, we have that $\tilde\bu$ is a weak solution of the 2D NSE \eqref{nse} on $I$.
By uniqueness of solutions, it follows that $\tilde\bu = S_\nu(\bu_0)$.  Then, by a contradiction argument, we obtain that in fact the entire sequence $\{S_\nu(\bu_{0,n})\}_n$ converges to $S_\nu(\bu_0)$ in $\mC(J,\Wrw)$.  
This concludes the proof.
\end{proof}

To verify assumptions \ref{Main2:H3} and \ref{Main2:H4}, we fix $\nu_0 > 0$ and introduce the following auxiliary space. Let $R>0$ and $J\subseteq I$ be an interval  closed and bounded on the left with left endpoint $t_0$, and consider the following inequalities for $\bu \in \mC(J,\Wrw)$:
\begin{align}\label{ineq:sup:Wrw}
\|\nabla^\perp\cdot\bu(t)\|_{L^r} \leq
	\left(R^r + (\nu_0\lambda_1)^{1-r}\|\nabla^\perp \cdot \f\|^r_{L^r(t_0,t;L^r)} \right)^{1/r} e^{(r-1)\nu_0\lambda_1(t-t_0)/r},
\end{align}
and
\begin{multline}\label{ineq:Dt:L2Vp}
	\|\partial_t \bu\|_{L^2(t_0,t;V')} 	\leq c\lambda_1^{-1/2} \|\f\|_{L^2(t_0,t;H)} + \\ 
	c\lambda_1^{-1/2+1/r} \left[ \nu_0 +\left(R^2 + \frac{1}{\nu_0 \lambda_1} \|\f\|_{L^2(t_0,t;H)}^2  \right)e^{\nu_0 \lambda_1(t - t_0)} \right]\times\\ \left( R^r + (\nu_0\lambda_1)^{1-r}\|\nabla^\perp \cdot \f\|_{L^r(t_0,t;L^r)}^r \right)e^{(r-1)\nu_0\lambda_1(t-t_0)},
\end{multline}
for $t\in J$, where $c>0$ is a universal constant. Then, we define 
\begin{align}
	\mY_{J}(R) = \left\{ \bu \in \mC(J,\Wrw) \,:\,\bu \mbox{ satisfies } \eqref{ineq:sup:Wrw} \mbox{ and } \eqref{ineq:Dt:L2Vp} \mbox{ for all } t\in J \right\}.	
 \label{def:YJR2d}
\end{align}

Note that, for all $0 < \nu \leq \nu_0$ and for every initial datum in $B_{\Wr}(R)$, the restriction to $J$ of the corresponding weak solution of the 2D $\nu$-NSE belongs to $\mY_{J}(R)$.

We observe that given any sequence of compact subintervals $J_n \subset I$, $n \in \NN$, each with left endpoint $t_0$ and such that $I = \bigcup_n J_n$, then
\begin{align}\label{eq:YI:R:Jn:2d}
	\mY_I(R) =  \bigcap_{n=1}^\infty \Pi_{J_n}^{-1} \mY_{J_n}(R),
\end{align}
where $\Pi_{J_n}:  \mC(I,\Wrw)\rightarrow  \mC(J_n,\Wrw)$ denotes the restriction operator on $J_n$ defined in \eqref{def:restriction:op}. We now show that this auxiliary space is compact.

\begin{lemma}\label{YJRcompact}
Let $R > 0$ and let $J \subset I$ be a compact subinterval with left endpoint $t_0$. Then, $\mY_{J}(R)$ is a compact subset of $\mC(J,\Wrw)$. Consequently, $\mY_I(R)$ is a compact subset of $\mC_{\rloc}(I,\Wrw)$.
\end{lemma}
\begin{proof}
First, from the definition of $\mY_J(R)$ in \eqref{def:YJR2d} it follows that, for all $\bu \in \mY_J(R)$, it holds
\begin{align}\label{ineq:gradp:u:RJ}
	\|\nabla^\perp\cdot\bu(t)\|_{L^r} \leq R_J \quad \mbox{ and } \quad \| \partial_t \bu\|_{L^2(t_0,t;V')} \leq \tilde{R}_J, \quad \mbox{for all } t \in J,
\end{align}
where $R_J, \tilde{R}_J$ are positive constants which depend on $J$, but are independent of $t$. In particular, the first inequality in \eqref{ineq:gradp:u:RJ} implies that $\mY_J(R) \subset \mC(J, B_{\Wr}(R_J)_\rw)$, so that $\mY_J(R)$ is metrizable, and it suffices to show that it is sequentially compact. 

Let $\{\bu_n\}_n$ be a sequence in $\mY_J(R)$. Then, $\{\bu_n\}_n$ is uniformly bounded in $\mC(J,\Wr)$ and $\{\partial_t \bu_n\}_n$ is uniformly bounded in $L^2(t_0,t,V')$ for all $t \in J$. With a similar argument as in the proof of \cref{contSnu}, it follows that $\{\bu_n\}_n$ is relatively compact in $\mC(J,\Wrw)$. Then, we can show that there exists a subsequence $\{\bu_{n'}\}_{n'}$ and $\bu \in \mC(J,\Wrw)$ with $\partial_t \bu \in L^2(t_0,t,V')$ such that
\begin{gather}
	\bu_{n'} \to \bu \mbox{ in } \mC(J,\Wrw), \label{conv:unp}\\
	\partial_t \bu_{n'} \rightharpoonup \partial_t \bu \mbox{ in } L^2(t_0,t; V'), \quad \mbox{ for all } t \in J. \label{conv:der:unp}
\end{gather}
To see this, first let $\{\bu_{n'}\}_{n'}$ be a subsequence for which $\bu_{n'} \to \bu$ in $\mC(J,\Wrw)$. Then, consider a sequence of points $\{t_k\}_{k \in \NN} \subset J$ that is dense in $J$. We have that $\{\partial_t \bu_{n'}\}$ is uniformly bounded in $L^2(t_0,t_k;V')$ for all $k$. Then, by a diagonalization argument, we may construct a further subsequence of $\{\bu_{n'}\}_{n'}$, which we still denote as $\{\bu_{n'}\}_{n'}$ for simplicity, such that $\partial_t \bu_{n'} \rightharpoonup \partial_t \bu$ in $L^2(t_0,t_k;V')$ for all $k$. Due to the continuity of the functional $t \mapsto \int_{t_0}^t \langle \partial_t \bu_{n'} - \partial_t \bu, \bv \rangle_{V',V} ds$ for any $\bv \in L^2(J;V')$, and the density of $\{t_k\}$ in $J$, we thus obtain that in fact $\partial_t \bu_{n'} \rightharpoonup \partial_t \bu$ in $L^2(t_0,t;V')$ for all $t \in J$.

With the convergences in \eqref{conv:unp} and \eqref{conv:der:unp}, we can pass to the limit in the inequalities from the definition of $\mY_J(R)$ and conclude that $\bu \in \mY_J(R)$. This concludes the proof of the compactness of $\mY_J(R)$. 

Consequently, in view of the characterization \eqref{eq:YI:R:Jn:2d}, it follows by employing again a standard diagonalization argument that $\mY_I (R)$ is compact in $\mC_{\rloc}(I,\Wrw)$.
\end{proof}

Next, we invoke \cref{YJRcompact} to verify assumption \ref{Main2:H3}.

\begin{prop}\label{SnuKinmK}
Let $K$ be a compact set in $\Wrw$, $2 \leq r < \infty$. Then, there exists a compact set $\mK \subset \mC_{\rloc}(I,\Wrw)$ such that
\begin{align*}
	S_\nu (K) \subset \mK, \quad \mbox{ for all } \nu \in (0,\nu_0].
\end{align*}
\end{prop}
\begin{proof}
Let $R > 0$ be such that $K \subset B_{\Wr}(R)$ and let $\{J_n\}_n$ be any sequence of compact subintervals of $I$, each with left endpoint $t_0$, such that $I = \bigcup_{n=1}^\infty J_n$.
From the estimates \eqref{estapriori2} and 
\eqref{estapriori3} in \cref{propestapriori}, it follows that $\Pi_{J_n} S_\nu(K) \subset \mY_{J_n}(R)$ for all $n\in \mathbb N$ and for every $\nu \in (0,\nu_0]$. 	
Thus, from the characterization \eqref{eq:YI:R:Jn:2d}, we deduce that $S_\nu(K)$ is contained in the compact set $\mK = \mY_I(R)$ for all $\nu \in (0,\nu_0]$.
\end{proof}

Finally, we verify that assumption \ref{Main2:H4} from \cref{Main2alt} is satisfied.

\begin{prop}\label{cor:conditions2DNS}
Let $K$ be a compact set in $\Wrw$, $2 \leq r < \infty$. Then,
\[
	\limsup_{\nu \to 0} S_\nu(K) \subset \mU_I,
\]
with $\mU_I$ as defined in \eqref{ue2d}. 
\end{prop}

\begin{proof}
From the proof of \cref{SnuKinmK}, we know that 
\begin{align*}
	\limsup_{\nu \to 0} S_\nu(K) \subset \mY_I(R),
\end{align*}
for any $R> 0$ such that $K \subset B_{\Wr}(R)$. Since $\mY_I(R)$ is a metrizable space, given $\bu \in \limsup_{\nu \to 0} S_\nu(K),$ there exists a sequence $\{\bu^{\nu_j}\}_j$ such that $\nu_j \in (0,\nu_0]$ and $\bu^{\nu_j} \in S_{\nu_j}(K)$
for all $j \in \mathbb{N}$, with $\nu_j \to 0$ and $\bu^{\nu_j} \to \bu$ in $\mY_I (R)$ as $j \to \infty$. Moreover, since the metric in $\mY_I(R)$ is compatible
with the topology in $\mC_{\rloc}(I,\Wrw)$, we also have that $\bu^{\nu_j} \to \bu$ in $\mC_{\rloc}(I,\Wrw)$ as $j \to \infty$. By \cref{thmNStoEuler},
this implies that $\bu \in \mU_I$, as desired.
\end{proof}

Having verified all the required assumptions, we may now apply \cref{Main2alt} with the choices of $X$, $\mU$,  $\{S_\ve\}_{\ve\in \mathcal E}$, and $\{\mu_\ve\}_{\ve\in \mathcal E}$ fixed in the beginning of this section and obtain the following result on the convergence of trajectory statistical solutions of the 2D Navier-Stokes equations to a trajectory statistical solution of the 2D Euler equations in the inviscid limit. 

\begin{theorem}\label{MainThmApp2D} Let $I \subset \mathbb R$ be an interval closed and bounded on the left with left endpoint $t_0$ and assume $\f\in L_{\rloc}^2(I, V').$ Fix $r \in [2,\infty)$, and let $S_\nu: \Wrw \to \Cloc(I,\Wrw)$, $\nu > 0$, and $\mU_I \subset \Cloc(I,\Wrw)$ be defined as in \eqref{def:Snu} and \eqref{ue2d}, respectively. Then, given a Borel probability measure $\mu_0$ on $\Wr$, the net $\{S_\nu\mu_0\}_{\nu > 0}$ has a subnet that converges as $\nu \to 0$,  with respect to the weak-star semicontinuity topology, to a $\mU_I$-trajectory statistical solution $\rho$ of the 2D Euler equations that satisfies $\Pi_{t_0}\rho=\mu_0$. 
\end{theorem}

Clearly, \cref{MainThmApp2D} thus yields the existence of a $\mU_I$-trajectory statistical solution of the 2D Euler equations satisfying a given initial datum. To emphasize this fact, we state it as the following corollary.

\begin{cor}
Let $I \subset \mathbb R$ be an interval closed and bounded on the left with left endpoint $t_0$ and assume $\f\in L_{\rloc}^2(I, V').$ Fix $r \in [2,\infty)$, and let $\mU_I \subset \Cloc(I,\Wrw)$ be as defined in \eqref{ue2d}. Then, given a Borel probability measure $\mu_0$ on $\Wr$, there exists a $\mU_I$-trajectory statistical solution $\rho$ of the 2D Euler equations satisfying the initial condition $\Pi_{t_0}\rho=\mu_0$.
\end{cor}

\subsection{Convergence of statistical solutions of 3D Navier-Stokes to 3D Euler} \label{subsec:3D:NSE:Euler}

In this section, we show the existence of a trajectory statistical solution of the 3D Euler equations starting from any given initial measure $\mu_0$ on $H$. This is done by considering, for each $\nu > 0$, a trajectory statistical solution $\rho_\nu$ of the 3D NSE starting from $\mu_0$, and constructing a suitable family of compact sets for which the assumptions of \cref{Main1} above are verified.

Here we consider the setting of \cref{Main1} with the following choices. Let $X = H_\rw$, and let $I\subset \mathbb R$ be an interval closed and bounded on the left with left endpoint $t_0$. Moreover, for a fixed forcing term $\f\in L^2_{\rloc}(I,H)$, we define the following corresponding sets of solutions of the 3D $\nu$-Navier-Stokes equations, with $\nu > 0$, and 3D Euler equations, in $\mX = \Cloc(I,H_\rw)$:
\begin{eqnarray}\label{def:U:nse:3d}
	\mU^\nu_I=\left\{\begin{array}{cc}\bunu \in \Cloc(I,H_\rw):& \bunu \mbox{ is a weak solution of the 3D $\nu$-Navier-Stokes}\\
		&\mbox{ equations \eqref{nse} on  $I$ in the sense of \cref{3NSweak}}
	\end{array}\right\},
\end{eqnarray}
and
\begin{eqnarray}\label{def:U:Euler:3d}
	\UE_I
	=\left\{\begin{array}{cc} \bu\in \Cloc(I,H_\rw): & \bu \mbox{ is a dissipative solution of the 3D Euler equations}\\
		&  \eqref{ee} \mbox{ on $I$ in the sense of \cref{Edissipative}}\end{array}\right\}.
\end{eqnarray}

As in \cref{subsec:2D:NSE:Euler}, we also define an auxiliary space in view of the a priori estimates from \cref{prop:apriori:est:3D}. Namely, for fixed $\nu_0 > 0$, $R > 0$, and any subinterval $J \subseteq I$ that is closed and bounded on the left with left endpoint $t_0$, consider the following inequalities for $\bu \in \Cloc(J,H_\rw)$:
\begin{align}\label{ineq:YI:1}
	 |\bu(t)|^2 \leq \left(R^2 + \frac{1}{\nu_0 \lambda_1} \|\f\|_{L^2(t_0, t; H)}^2 \right) e^{\nu_0 \lambda_1 (t-t_0)},
\end{align}
for $t \in J$, and
\begin{align}\label{ineq:YI:2}
	\| \bu(t) - \bu(s)\|_{\Winfp} 
	\leq c |t-s|^{1/2} \nu_0^{1/2} \lambda_1^{-3/4}  e^{\frac{\nu_0 \lambda_1(t - t_0)}{2}}  \left[ R^2 + \frac{1}{\nu_0 \lambda_1}  \|\f\|_{L^2(t_0,t;H)}^2  \right]^{1/2} \notag \\
	+ |t - s| e^{\nu_0 \lambda_1(t - t_0)}  \left[ R^2 + \frac{1}{\nu_0 \lambda_1}  \|\f\|_{L^2(t_0,t;H)}^2  \right],  
\end{align}
for $s,t \in J$ with $s < t$, where $c > 0$ is a fixed universal constant. Then, we define
\begin{align}\label{def:YIR:3D}
	\mY_{J}(R) =  \left\{ \bu \in \Cloc(J,H_\rw) : \bu \mbox{ satisfies } \eqref{ineq:YI:1} \mbox{ for all $t \in J$ and } \eqref{ineq:YI:2} \mbox{ for all $s < t$ in $J$} \right\}.
\end{align}

As in the two-dimensional case, note that for all $0 < \nu \leq \nu_0$ and for every initial datum in $B_{H}(R)$, the restriction to $J$ of the corresponding weak solution of the 3D $\nu$-NSE belongs to $\mY_{J}(R)$.

From the definition \eqref{def:YIR:3D}, it follows that for any nondecreasing sequence of compact subintervals $J_n \subset I$, $n \in \NN$, each with left endpoint $t_0$ and such that $I = \bigcup_n J_n$, we may write
\begin{align}\label{eq:YI:R:Jn}
	\mY_I(R) =  \bigcap_{n=1}^\infty \Pi_{J_n}^{-1} \mY_{J_n}(R),
\end{align}
where we recall from \eqref{def:restriction:op} that $\Pi_{J_n}$ denotes the operator that takes any function $\bu \in \Cloc(I,H_\rw)$ to its restriction to $J_n$, namely $\Pi_{J_n}\bu = \bu|_{J_n} \in \mC(J_n, H_\rw)$.

We have the following compactness result.

\begin{lemma}\label{3YJRcompact}
	Let $R> 0$ and $\nu_0 > 0$. Then, for every compact subinterval $J \subset I$ with left endpoint $t_0$, $\mY_J(R)$ is a compact subset of $\mC(J,H_\rw)$. Consequently, $\mY_I(R)$ is a compact subset of $\mC_{\rloc}(I,H_\rw)$.
\end{lemma}

\begin{proof}
Let $J \subset I$ be a compact subinterval with left endpoint $t_0$. Denote
\begin{align*}
	R_J = \left(R^2 + \frac{1}{\nu_0 \lambda_1} \|\f\|_{L^2(J; H)}^2 \right)^{1/2} e^{\frac{\nu_0 \lambda_1}{2} |J| }.
\end{align*}
According to the definition in \eqref{def:YIR:3D}, it follows that every $\bu \in \mY_J(R)$ satisfies:
\begin{align}\label{ineq:u:RJ}
	|\bu(t)| \leq R_J, \quad \mbox{ for all } t \in J,
\end{align}
and 
\begin{align}\label{ineq:diff:u:RJ}
	\|\bu(t) - \bu(s) \|_{\Winfp} &\leq c |t-s|^{1/2} \nu_0^{1/2} \lambda_1^{-3/4} R_J + |t-s| R_J^2  \notag \\
	&\leq C |t-s|^{1/2}	\quad \mbox{ for all } s,t \in J.
\end{align}
In particular, \eqref{ineq:u:RJ} implies that $\mY_J(R) \subset  \mC(J,B_{H}(R_J)_\rw)$. Thus, $\mY_J(R)$ is metrizable, and it suffices to show that $\mY_J(R)$ is sequentially compact. 

Let $\{\bu_n\}_n$ be a sequence in $\mY_J(R)$. We first show that $\{\bu_n\}_n$ is equicontinuous in $\mC(J,B_{H}(R_J)_\rw)$. Let $\bv \in H$ and $\varepsilon > 0$ be arbitrary. Since $\Winf$ is dense in $H$, we may take $\bw \in \Winf$ such that $|\bv  - \bw| < \varepsilon/(4 R_J)$. Thus, together with \eqref{ineq:u:RJ} and \eqref{ineq:diff:u:RJ}, we obtain that, for any $s,t \in J,$
\begin{align*}
	|(\bu_n(t) - \bu_n(s), \bv)| 
	&\leq  |(\bu_n(t) - \bu_n(s), \bv - \bw)| + |\langle \bu_n(t) - \bu_n(s), \bw \rangle_{\Winfp,\Winf}| \\
	&\leq 2 R_J |\bv - \bw| + \| \bu_n(t) - \bu_n(s) \|_{\Winfp} \|\bw \|_{\Winf} \\
	&< \frac{\varepsilon}{2}+ C |t-s|^{1/2} \|\bw \|_{\Winf},
\end{align*} 
so that, if $|t-s| < \varepsilon^2 / (2C \| \bw\|_{\Winf})^2,$ then
\begin{align*}
	|(\bu_n(t) - \bu_n(s), \bv)| < \varepsilon, \quad \mbox{ for all }n.
\end{align*}
Since $\bv$ and $\varepsilon$ are arbitrary, we deduce that $\{\bu_n\}_n$ is equicontinuous in $\mC(J,B_{H}(R_J)_\rw)$. Moreover, for each fixed $t \in J$, $\{\bu_n(t)\}_n \subset B_H(R_J)$, and thus $\{\bu_n(t)\}_n $ is relatively compact in $B_H(R_J)_\rw$. Therefore, by the Arzel\`a-Ascoli theorem, it follows that $\{\bu_n\}_n$ is relatively compact in $\mC(J,B_{H}(R_J)_\rw)$, and hence in $\mC(J,H_\rw)$.

Thus, there exists a subsequence $\{\bu_{n'}\}_{n'}$ of $\{\bu_{n}\}_n$ and $\bu \in \mC(J,H_\rw)$ such that $\bu_{n'} \to \bu$ in $\mC(J,H_\rw)$. It only remains to show that $\bu \in \mY_J(R)$. The inequality \eqref{ineq:YI:1} follows immediately from the weak convergence $\bu_{n'}(t) \to \bu(t)$ in $H_\rw$. Moreover, denoting the right-hand side of \eqref{ineq:YI:2} by $K(s,t)$, it follows that for every $\bw \in \Winf$ such that $\|\bw\|_{\Winf} \leq 1$ and for all $s,t \in J$ with $s < t$, we have
\begin{multline*}
	|\langle \bu(t) - \bu(s), \bw \rangle_{\Winfp, \Winf} |  
	= | (\bu(t) - \bu(s), \bw ) | \\
	= \lim_{n' \to \infty}  | (\bu_{n'}(t) - \bu_{n'}(s), \bw ) | 
	= \lim_{n' \to \infty} |\langle \bu_{n'}(t) - \bu_{n'}(s), \bw \rangle_{\Winfp, \Winf} | 
	\leq K(s,t).
\end{multline*}
This implies that 
\begin{align*}
	\| \bu(t) - \bu(s) \|_{\Winfp} \leq K(s,t),
\end{align*}
and hence $\bu$ satisfies \eqref{ineq:YI:2}. Consequently, $\bu \in \mY_J(R)$, as desired.

The second part of the statement follows from the characterization in \eqref{eq:YI:R:Jn}, and the fact that each $\mY_{J_n}(R)$ is a compact set in $\mC(J_n,H_\rw)$. By a diagonalization argument, we deduce that $\mY_I(R)$ is compact in $\mC_{\rloc}(I,H_\rw)$. This concludes the proof.
\end{proof}
	
We now obtain the following result as an application of \cref{Main1}. It shows convergence of a vanishing viscosity net of $\mU_I^\nu$-trajectory statistical solutions of the 3D NSE towards a $\UE_I$-trajectory statistical solution of the 3D Euler equations. 

\begin{theorem}\label{MainThmApp3D}
	Let $I \subset \RR$ be an interval closed and bounded on the left with left endpoint $t_0$ and assume that $\f\in L^2_{\rloc}(I,H).$ Consider $\mU^\nu_I, \UE_I \subset \Cloc(I,H_\rw)$, $\nu > 0$, as defined in \eqref{def:U:nse:3d} and \eqref{def:U:Euler:3d}, respectively. Also, let $\mu_0$ be a Borel probability measure on $H$, and, for each $\nu > 0$, let $\rho_\nu$ be a $\mU^\nu_I$-trajectory statistical solution of the 3D $\nu$-NSE such that $\Pi_{t_0} \rho_\nu = \mu_0$. Then, there exists a subnet of $\{\rho_\nu\}_{\nu > 0}$ that converges as $\nu \to 0$, with respect to the weak-star semicontinuity topology, to a $\UE_I$-trajectory statistical solution $\rho$ of the 3D Euler equations satisfying $\Pi_{t_0}\rho = \mu_0$.
\end{theorem}
\begin{proof}
As in the beginning of this subsection, let us denote $X = H_\rw$ and $\mX = \Cloc(I,H_\rw)$. We proceed by constructing a sequence of compact sets $\{\mK_n\}_{n \in \NN}$ in $\mX$ which satisfies the assumptions of \cref{Main1}. First, since $H$ is a separable Banach space then, as recalled in \cref{subsec:meas:theory}, it follows that $\mu_0$ is tight. Moreover, the Borel $\sigma$-algebras in $H$ and $H_\rw$ coincide. 
Thus, given any sequence of positive real numbers $\{\delta_n\}_{n \in \NN}$ with $\delta_n \to 0$, there exists a corresponding sequence of (strongly) compact sets $\{K_n\}_{n \in \NN}$ in $H$ such that
\begin{align}\label{mu0:Kn}
	\mu_0(X \backslash K_n) < \delta_n, \quad \mbox{for all } n.
\end{align}

For each $n$, let $R_n > 0$ such that $K_n \subset B_H(R_n)$, and define 
\begin{align*}
	\mK_n = \mY_I(R_n) \cap \Pi_{t_0}^{-1} K_n,
\end{align*}
with $\mY_I(R_n)$ as defined in \eqref{def:YIR:3D}, for fixed $\nu_0 > 0$. From \cref{3YJRcompact}, $\mY_I(R_n)$ is a compact set in $\mX$. Moreover, since $K_n$ is compact in $H$, hence also compact (and closed) in $X = H_\rw$, and $\Pi_{t_0}: \mX \to X$ is a continuous operator, then $\Pi_{t_0}^{-1} K_n$ is a closed set in $\mX$. This implies that each $\mK_n$ is a compact set in $\mX$.

Let us verify that the sequence $\{\mK_n\}_n$ satisfies condition \ref{i:Main1} of \cref{Main1}. From \cref{def-stat-sol}, for each $\nu > 0$ there exists a Borel set $\mV_\nu$ in $\mX$ such that $\mV_\nu \subset \mU_I^\nu$ and $\rho_\nu(\mX \backslash \mV_\nu) = 0$\footnote{In fact, since, as shown in \cite[Proposition 2.12]{FRT2013}, $\mU_I^\nu$ is itself a Borel set in $\mX$, then we could take $\mV_\nu = \mU_I^\nu$.}. Moreover, from \cref{prop:apriori:est:3D} and the definition of $\mY_I(R_n)$ in \eqref{def:YIR:3D} and the fact that $\mU_I^\nu \subset \mY_I(R_n)$ when $\nu \leq \nu_0$, it follows that
\begin{align*}
	\mV_\nu \cap \Pi_{t_0}^{-1} K_n  \subset \mU_I^\nu \cap \Pi_{t_0}^{-1} K_n \subset \mY_I(R_n) \cap \Pi_{t_0}^{-1} K_n = \mK_n, \quad \mbox{for all } 0 < \nu \leq \nu_0.
\end{align*}
With these two facts, we obtain that, for all $0< \nu \leq \nu_0,$
\begin{align*}
	\rho_\nu(\mX \backslash \mK_n) 
	&\leq \rho_\nu (\mX \backslash ( \mV_\nu \cap \Pi_{t_0}^{-1} K_n )) \\
	&\quad= \rho_\nu (\mX \backslash \Pi_{t_0}^{-1} K_n)
	= \rho_\nu (\Pi_{t_0}^{-1} (X \backslash K_n))
	= \Pi_{t_0} \rho_\nu (X \backslash K_n)
	= \mu_0 (X \backslash K_n).
\end{align*}
Thus, together with \eqref{mu0:Kn}, we deduce that
\begin{align*}
	\rho_\nu(\mX \backslash \mK_n) < \delta_n, \quad \mbox{for all } n \,\mbox{ and } \, 0 < \nu \leq \nu_0,
\end{align*}
as desired.

Regarding condition \ref{ii:Main1} of \cref{Main1}, first note that
\begin{align*}
	\mV_\nu \cap \mK_n  \subset \mU_I^\nu \cap \mK_n  \subset \Cloc(I, B_H(R_n)_\rw)
\end{align*}
for all $\nu > 0$ and $n \in \NN$. Thus,
\begin{align*}
	\limsup_{\nu \to 0} (\mV_\nu \cap \mK_n)  \subset \Cloc(I, B_H(R_n)_\rw), \quad \mbox{for all } n.
\end{align*}
Since $\Cloc(I, B_H(R_n)_\rw)$ is metrizable, given $\bu \in \limsup_{\nu \to 0} (\mU_I^\nu \cap \mK_n)$ there exists a sequence $\{\bu_{\nu_j}\}_{j \in \NN}$ such that $\bu_{\nu_j} \in \mV_\nu \cap \mK_n \subset \mU_I^\nu \cap \Pi_{t_0}^{-1} K_n$ for all $j \in \NN$, with $\nu_j \to 0$ and $\bu_{\nu_j} \to \bu$ in $\Cloc(I,H_\rw)$ as $j \to \infty$. By \cref{thm3dNStoEuler}, this implies that $\bu \in \mU^{\diss}_I$. Hence, condition \ref{ii:Main1} of \cref{Main1} is satisfied. The conclusion now follows as an application of \cref{Main1}.
\end{proof}

Given any Borel probability measure $\mu_0$ on $H$ and $\nu > 0$, existence of a $\mU_I^\nu$-trajectory statistical solution $\rho_\nu$ of the 3D NSE in the sense of \cref{def-stat-sol} satisfying the initial condition $\Pi_{t_0}\rho_\nu = \mu_0$ is shown in \cite[Theorem 4.2]{BMR2016} (see also \cite{FRT2013}). This fact together with \cref{MainThmApp3D} thus yields the following corollary on the existence of $\UE_I$-trajectory statistical solutions of the 3D Euler equations satisfying a given initial datum.

\begin{cor}
	Let $I \subset \mathbb R$ be an interval closed and bounded on the left with left endpoint $t_0$ and assume that $\f\in L^2_{\rloc}(I,H).$ Let $\UE_I \subset \Cloc(I,H_\rw)$ be as defined in \eqref{def:U:Euler:3d}. Then, given a Borel probability measure $\mu_0$ on $H$, there exists a $\UE_I$-trajectory statistical solution $\rho$ of the 3D Euler equations satisfying the initial condition $\Pi_{t_0}\rho=\mu_0$.
\end{cor}

\subsection{Convergence of statistical solutions of the Galerkin approximations of the 3D NSE}
\label{subsec:3D:Galerkin:NSE} 

We now address the Galerkin approximation of the three-dimensional Navier-Stokes equations \eqref{subsec:setting} on a periodic spatial domain $\Omega$, on a time interval $I\subset\RR$ closed and bounded on the left, with left endpoint $t_0\in\RR$, and with $\f\in L^2_{\rloc}(I,H)$. Our aim is to apply \cref{Main2alt} to show that trajectory statistical solutions generated by the well-defined solution operator of the Galerkin approximations converge to a trajectory statistical solution of the 3D NSE. For the framework of \cref{subsubsec:3D}, the phase space is taken to be $X = H_\rw,$ so the trajectory space is $\mX=\Cloc(I,H_\rw)$. The set $\mU^\nu_I$ is that of the weak solutions of the 3D $\nu$-NSE on $I$ defined in \eqref{def:U:nse:3d}.

For each $m \in \NN$, let $P_m$ denote the projection of $H$ onto the finite-dimensional subspace of $H$ spanned by the first $m$ eigenfunctions $\bw_1,\ldots, \bw_m$ of the Stokes operator which, on the periodic case, coincides with the negative Laplacian operator $-\Delta$ on $V \cap H^2(\Omega)^3$. The Galerkin approximation in $P_m H$ of the 3D NSE \eqref{nse} is defined as 
\begin{align}\label{eq:Galerkin:nse}
	\partial_t \bu_m - \nu \Delta \bu_m + P_m [(\bu_m \cdot \nabla) \bu_m ] + \nabla p_m = P_m \f, \quad \nabla \cdot \bu_m  = 0,
\end{align}
see e.g. \cite{temam1995} for more details. 

By taking the inner product of the first equation in \eqref{eq:Galerkin:nse} with each of the eigenfunctions $\bw_1, \ldots, \bw_m$ and writing $\bu_m(t) = \sum_{i=1}^m \alpha_{i m}(t) \bw_i$, it follows that the Galerkin approximation is equivalent to a system of $m$ ordinary differential equations on $I$ of the form $\balpha_t = \bF(t, \balpha)$, where $\balpha_m = (\alpha_{1m}, \ldots, \alpha_{mm})$. The right-hand side $\bF$ is quadratic (hence locally Lipschitz) in $\balpha$ and, due to $\f \in  L^2_{\rloc}(I,H)$, $\bF$ is also measurable in $t$ and bounded by an integrable function of $t$ on every compact subset of $I \times \RR^m$. As such, we obtain, from the classical Carath\'eodory theory of existence and uniqueness of solutions for ODEs \cite{CoddingtonLevinson1987,Hale1980}, an absolutely continuous function $\bu_m$ on $[t_0, t_m]$, for some $t_m \in I$, which satisfies \eqref{eq:Galerkin:nse} a.e. in $t$ and a given initial condition $\bu_m(t_0) = \bu_{0,m} \in P_m H$. Moreover, from standard energy estimates, it follows that $\bu_m$ satisfies the following inequality for any compact subinterval $J \subset I$ with left endpoint $t_0$ and containing $[t_0, t_m]$:
\begin{align*}
	|\bu_m(t)|^2 \leq |\bu_{0,m}|^2 e^{-\nu \lambda_1 (t-t_0)} + \frac{1}{\nu \lambda_1} \|\f\|^2_{L^2(J,H)} 
\end{align*}
for every $t \in [t_0, t_m]$. This implies that $\bu_m$ is uniformly bounded in $t$, and consequently $\bu_m,$ in fact, exists and is unique for all $t \in I$, with $\bu_m \in \mC_{\rloc}(I,P_m H)$. 

Therefore, we can define a solution operator associated to \eqref{eq:Galerkin:nse}, given as
\begin{eqnarray}\label{def:Sm:Gal}
	S_m: H_\rw  &\rightarrow &\mC_{\rloc}(I,P_m H) \subset \mC_{\rloc}(I,H_\rw) \nonumber \\
	\bu_0 &\mapsto& \bu_m,
\end{eqnarray}
where $\bu_m = \bu_m(t)$ is the unique trajectory solving \eqref{eq:Galerkin:nse} on $I$ subject to the initial condition $\bu_m(t_0) = P_m \bu_0$. 

For the statistical solutions, we consider a sequence of initial measures $\{\mu_{0,m}\}_{m\in\NN}$ on $H_\rw$ associated with the Galerkin approximations. Since the initial conditions associated with the Galerkin operator $S_m$ belong to $P_m H,$ it is natural to assume that each $\mu_{0,m}$ is a Borel probability measure on $H_\rw$ (or, equivalently, on $H$) which is carried by $P_m H.$ For the sake of convergence, we also assume that these measures converge to a Borel probability measure on $H,$ in the sense of weak-star semicontinuity topology.

With this setting, we have the following convergence result.

\begin{theorem}\label{thm:conv:ss:Gal}
	Let $\nu > 0$ and let $I \subset \mathbb R$ be an interval that is closed and bounded on the left with left endpoint $t_0$. Assume that $\f \in  L^2_{\rloc}(I,H).$ Let $S_m: H_\rw \to \Cloc(I,P_m H)$, $m \in \NN$, and $\mU^\nu_I \subset \Cloc(I,H_\rw)$ be defined as in \eqref{def:Sm:Gal} and \eqref{def:U:nse:3d}, respectively. Let $\mu_0$ be a (tight) Borel probability measure on $H_\rw.$ Suppose $\{\mu_{0,m}\}_{m \in \NN}$ is a sequence of (tight) Borel probability measures on $H_\rw$ carried by $P_m H$ which is uniformly tight on $H_\rw$ and converges to $\mu_0$ in the sense of weak-star semicontinuity topology, i.e. $\mu_{0,m} \wconv \mu_0$ in $H_\rw.$ Then, the sequence of measures $\{S_m \mu_{0,m}\}_{m \in \NN}$ has a subsequence that converges, as $m \to \infty$, with respect to the weak-star semicontinuity topology, to a $\mU^\nu_I$-trajectory statistical solution $\rho$ of the 3D Navier-Stokes equations satisfying $\Pi_{t_0}\rho=\mu_0$.
\end{theorem}

\begin{proof}
    We proceed to verify that assumptions \ref{Main2:H0}-\ref{Main2:H4} of \cref{Main2alt} hold under this setting. 
    
    First, from the definition \eqref{def:Sm:Gal} of the Galerkin semigroup as the weak solution of the Galerkin approximation \eqref{eq:Galerkin:nse} with the initial condition $\bu_m(t_0) = P_m \bu_0,$ it follows that the operator $\Pi_{t_0} S_m$ considered in the statement of \cref{Main2alt} is such that $\Pi_{t_0}S_m\bu_0 = \bu_m(t_0) = P_m\bu_0,$ so that this operator is precisely the Galerkin projector, i.e. $\Pi_{t_0} S_m = P_m.$ Thus, condition \ref{Main2:H0} reads $P_m\mu_{0,m} \wconv \mu_0.$ Since $\mu_{0, m}$ is assumed to be carried by $P_m H,$ we thus have $P_m\mu_{0, m} = \mu_{0, m}.$ Indeed, for any Borel set $A$ in $H$, we have
    \begin{align*}
        P_m \mu_{0,m} (A) = \mu_{0,m} (P_m^{-1} A)
        = \mu_{0,m} (P_m^{-1} A \cap P_m H) 
        = \mu_{0,m} (A \cap P_m H ) = \mu_{0,m} (A),
    \end{align*}
    where, in the second and fourth equalities, we used the fact that $\mu_{0,m}$ is carried by $P_m H,$ while, in the third equality, we used that $P_m$ is a projection operator, so that $P_m^{-1}A \cap P_m H = A\cap P_m H.$
    
    Thus, since $P_m\mu_{0, m} = \mu_{0, m},$ condition \ref{Main2:H0} is precisely the assumption that we have, i.e. that $\mu_{0,m} \wconv \mu_0.$ Hence, condition \ref{Main2:H0} is satisfied.
    
    From the equivalence between \eqref{eq:Galerkin:nse} and the system of $m$ ordinary differential equations $\balpha_t  = \bF(t, \balpha)$, together with the properties of $\bF$ recalled above, it follows again from classical ODE theory that any solution $\bu_m$ of \eqref{eq:Galerkin:nse} depends continuously on initial data. This implies that the solution operator $S_m: H_\rw  \to \mC_{\rloc}(I,P_m H) \subset \mC_{\rloc}(I,H_\rw)$ is continuous, and hence assumption \ref{Main2:H1} is verified.
    
    The validity of assumption \ref{Main2:H2} follows immediately from the condition that the sequence of initial measures $\{\mu_{0,m}\}_m$ is uniformly tight in $H_\rw$. Indeed, given any sequence $\delta_n \rightarrow 0,$ there is a corresponding sequence of compact sets $K_n$ in $H_\rw$, $n \in \NN$, such that \ref{Main2:H2} holds. Since, as we show next, the remaining assumptions hold for any compact set $K$ in $H_\rw$, they hold in particular for such sequence.

    To establish \ref{Main2:H3} and \ref{Main2:H4}, we first define an auxiliary space analogous to \eqref{def:YIR:3D}. Since, in the current setting, $\nu$ is a fixed parameter, we may invoke a different set of inequalities than \eqref{ineq:YI:1}-\eqref{ineq:YI:2} to define this auxiliary space, which yield estimates in stronger norms. These alternative inequalities are indeed necessary for guaranteeing compactness of such auxiliary space in the topology of $L^2_\rloc(I,H)$. This in turn allows us to obtain a result analogous to \cref{thm3dNStoEuler}, showing that individual solutions of the Galerkin approximations converge to a Leray-Hopf weak solution of 3D NSE. 
    
    More precisely, under the present framework, we have that for any $m \in \NN$ and any solution $\bu_m$ of \eqref{eq:Galerkin:nse} on $I$, the following inequalities hold:
        \begin{align}\label{ineq:energy:Gal}
        |\bu_m(t)|^2 + \nu \int_{t_0}^t \| \bu_m(\tau)\|^2 \rd \tau \leq |\bu_m(t_0)|^2 + \frac{1}{\nu \lambda_1 } \| \f \|^2_{L^2(t_0, t; H)},
    \end{align}
    for all $t \in I$, and 
    \begin{multline}\label{ineq:diff:um:t:s}
        \| \bu_m(t) - \bu_m(s) \|_{V'} \leq c \nu^{1/2} |t-s|^{1/2} \left( |\bu_m(t_0)|^2 + \frac{1}{\nu \lambda_1} \| \f \|^2_{L^2(t_0, t; H)} \right)^{1/2} \\
        + c \nu^{-3/4} |t-s|^{1/4}  \left( |\bu_m(t_0)|^2 + \frac{1}{\nu \lambda_1} \| \f \|^2_{L^2(t_0, t; H)} \right),
    \end{multline}
    for all $s,t \in I$ with $s \leq t$, and for some positive constant $c$ which is independent of $m$.
    
    The proof of \eqref{ineq:energy:Gal} follows from typical energy estimates, see e.g. \cite[Chapter 3]{temam1995}. Inequality \eqref{ineq:diff:um:t:s} then follows by proceeding similarly as in the proof of \eqref{apriori:est:2} in \cref{prop:apriori:est:3D} and invoking \eqref{ineq:energy:Gal}. We omit further details. 
    
    Now, given an arbitrary $R > 0$ and an arbitrary subinterval $J \subseteq I$ that is closed and bounded on the left with left endpoint $t_0$, consider the following inequalities for $\bu \in \mC_{\rloc}(J,H_\rw)$:
    \begin{align}\label{ineq:YI:Gal:1}
        |\bu(t)|^2 + \nu \int_{t_0}^t \| \bu(\tau) \|^2 \rd \tau \leq R^2 + \frac{1}{\nu \lambda_1 } \| \f \|^2_{L^2(t_0, t; H)},
    \end{align}
    for $t \in J$, and
    \begin{multline}\label{ineq:YI:Gal:2}
        \| \bu(t) - \bu(s) \|_{V'} \leq c \nu^{1/2} |t-s|^{1/2} \left( R^2 + \frac{1}{\nu \lambda_1} \| \f \|^2_{L^2(t_0, t; H)} \right)^{1/2} \\
        + c \nu^{-3/4} |t-s|^{1/4}  \left( R^2 + \frac{1}{\nu \lambda_1} \| \f \|^2_{L^2(t_0, t; H)} \right),
    \end{multline}
    for $s,t \in J$ with $s < t$, where $c$ is the same constant from \eqref{ineq:diff:um:t:s}. Based on these, we define the set 
    \begin{align}\label{def:YIR:Gal}
        \mY_{J}(R) =  \left\{ \bu \in \Cloc(J,H_\rw) : \bu \mbox{ satisfies } \eqref{ineq:YI:Gal:1} \mbox{ for all $t \in J$ and } \eqref{ineq:YI:Gal:2} \mbox{ for all $s < t$ in $J$} \right\}.
    \end{align}
    
    Then, the same characterization as in \eqref{eq:YI:R:Jn} holds for $\mY_I(R)$ in this case, and by analogous arguments as in \cref{3YJRcompact} we deduce that $\mY_I(R)$ is a compact subset of $\Cloc(I,H_\rw)$. Moreover, analogously as in \cref{SnuKinmK}, we can invoke the inequalities \eqref{ineq:energy:Gal} and \eqref{ineq:diff:um:t:s} to show that for any compact set $K$ in $H_\rw$ and any $R > 0$ such that $K \subset B_H(R)$, it holds that $S_m(K) \subset \mY_I(R)$ for all $m \in \NN$. This shows that assumption \ref{Main2:H3} is satisfied.
    
    Finally, to verify assumption \ref{Main2:H4}, we argue similarly to the the proof of \cref{cor:conditions2DNS}. Specifically, given a compact set $K$ in $H_\rw$, let $R > 0$ such that $K \subset B_H(R)$. As we showed in the verification of \ref{Main2:H3}, this implies that $S_m(K) \subset \mY_I(R)$ for all $m$, and hence $\limsup_m S_m(K) \subset \mY_I(R)$. Since $\mY_I(R)$ is metrizable, given $\bu \in \limsup_m S_m(K)$ there exists a sequence $\{\bu_{m_j}\}_{j \in \NN}$ such that $\bu_{m_j} \in S_{m_j}(K)$ for all $j$, and $\bu_{m_j} \to \bu$ in $\mY_I(R)$ as $j \to \infty$. Consequently, $\bu_{m_j} \to \bu$ in $\mC_\rloc(I, H_\rw)$ as $j \to \infty$. 
    
    Moreover, from the fact that $\bu_{m_j} \in S_{m_j}(K) \subset \mY_I(R)$ for all $j$ and from the definition of $\mY_I(R)$ in \eqref{def:YIR:Gal}, it follows that the sequence $\{\bu_{m_j}\}_j$ has the uniform upper bounds implied by \eqref{ineq:YI:Gal:1} and \eqref{ineq:YI:Gal:2} on every compact subinterval $J \subset I$ with left endpoint $t_0$. Then, standard compactness arguments yield that $\bu \in L^2_{\rloc}(I,V)$ and, modulo a subsequence, $\bu_{m_j} \to \bu$ in $L^2_{\rloc}(I,H)$. Combining all these facts, we may thus pass to the limit $m \to \infty$ in the weak formulation of \eqref{eq:Galerkin:nse} and deduce that $\bu$ is a Leray-Hopf weak solution of the 3D NSE, as defined in \cref{3NSweak}. See e.g. \cite[Chapter 3, Section 3]{temam1995} for similar arguments. Therefore, $\bu \in \mU^\nu_I$, and we conclude that $\limsup_m S_m(K) \subset \mU^\nu_I$. This shows that \ref{Main2:H4} holds.
    
    We have thus verified all the assumptions from \cref{Main2alt}, which then yields the desired result.
\end{proof}

\begin{remark}\label{rmk:mu0m:alt:cond}
	Note that if $\{\mu_{0,m}\}_m$ is a sequence of Borel probability measures on $H_\rw$ which converges to a Borel probability measure $\mu_0$ with respect to the weak-star semicontinuity topology in $H$ (with the strong topology), then it also converges in the weak-star topology in $H$, see \cref{portmanteau}. Since $H$ is a Polish space, then by Prohorov's theorem \cite{prohorov} the relatively compact subsets of $\mP(H)$ coincide with the uniformly tight ones. In particular, it follows that $\{\mu_{0,m}\}_m$ is uniformly tight in $H$, and consequently in $H_\rw$. Therefore, if $\mu_{0,m} \wconv \mu_0$ in $H$ for some $\mu_0 \in \mP(H)$ then the condition from \cref{thm:conv:ss:Gal} that $\{\mu_{0,m}\}_m$ is uniformly tight on $H_\rw$ is immediately satisfied.
\end{remark}

\begin{remark}\label{rmk:Pm:mu0}
    In the statement of \cref{thm:conv:ss:Gal}, the conditions imposed on the initial approximating measures $\mu_{0,m}$ are, in a sense, generic. In practice, one would want to start with something more specific. For example, given an initial tight Borel probability measure $\mu_0$ of interest for the limit problem, we may consider the Galerkin projections $\mu_{0,m} = P_m\mu_0$ of that measure. Note that $P_m\mu_0$ is carried by $P_m H_\rw.$ Let us verify that such approximating measures satisfy the remaining conditions of \cref{thm:conv:ss:Gal}.

	In order to see that $\{P_m \mu_0\}_m$ is uniformly tight in $H_\rw$, fix $\delta > 0 $ and let $K$ be a compact set in $H_\rw$ such that $\mu_0(H \backslash K) < \delta$. Let also $R > 0$ be such that $K \subset B_H(R)$. The subset $B_H(R)$ is a compact set in $H_\rw$, and 
	\begin{align*}
		P_m \mu_0 (H \backslash B_H(R)) = \mu_0(P_m^{-1} (H \backslash B_H(R))) =\mu_0 (H \backslash P_m^{-1} B_H(R)) \leq \mu_0(H \backslash B_H(R)) \\
		\leq \mu_0 (H \backslash K) < \delta 
	\end{align*}
	for all $m$, where in the first inequality we used that $P_m B_H(R) \subset B_H(R)$. This shows that $\{P_m \mu_0\}_m$ is uniformly tight in $H_\rw$. 
	
	The second condition is that $P_m \mu_0 \wconv \mu_0$ in $H_\rw.$ This can be seen by noting that $P_m \bu_0 \to \bu_0$ in $H,$ for all $\bu_0 \in H$, and hence, by the Dominated Convergence Theorem, we have that for every bounded and continuous real-valued function $\varphi$ on $H$
    \[
        \int_H \varphi(\bu) \rd(P_m \mu_0)(\bu) = \int_H \varphi(P_m \bu) \rd\mu_0(\bu) \to \int_H \varphi(\bu) \rd\mu_0(\bu),
    \]
    so that $P_m \mu_0 \wsconv \mu_0$ in $H$. Since $H$ is completely regular, it follows from \cref{portmanteau} that $P_m \mu_0 \wconv \mu_0$ in $H.$ Moreover, since any open set in $H_\rw$ is open in $H,$ we obtain from the equivalence between conditions \ref{PMwscstar} and \ref{PMlimopen} in \cref{portmanteau} that $P_m \mu_0 \wconv \mu_0$ in $H_\rw.$
\end{remark}

\begin{remark}\label{rmk:mu0m:montecarlo}
    Another useful practical example is with a Monte-Carlo approximation $\mu_0^N = (1/N) \sum_{k=1}^{N} \delta_{\bu_k}$, $\bu_k \in H,$ of a desired initial (tight) Borel probability measure $\mu_0$ on $H_\rw.$ The convergence $P_m\mu_{0, N_m} \wconv \mu_0$ for a suitable subsequence $\{N_m\}_m$ is a delicate issue, though, but it can be proved in some cases. For a related result for the two-dimensional Navier-Stokes equations and a Gaussian initial measure with the eigenvalues of the covariance operator decaying sufficiently fast, see \cite{DBLP:conf/mcqmc/BarthSS14}. This will be further discussed in subsequent works. 
\end{remark}

\begin{remark}\label{rmk:exist:ss:3D:NSE}
	As a byproduct of \cref{thm:conv:ss:Gal}, we obtain, for any given initial measure $\mu_0$, the existence of a $\mU^\nu_I$-trajectory statistical solution $\rho$ of the 3D Navier-Stokes equations satisfying the initial condition $\Pi_{t_0}\rho=\mu_0$. \cref{thm:conv:ss:Gal} thus provides an alternative proof of this fact to the one previously given in \cite[Theorem 4.2]{BMR2016} (see also \cite{FRT2013}), where existence was shown via an approximation by convex combinations of Dirac measures, by invoking the Krein-Milman theorem together with a tightness argument. Here, existence is derived instead via convergence of standard Galerkin approximations.
\end{remark}


\section*{Acknowledgements}
ACB received support under the grants \#2019/02512-5, S\~ao Paulo Research Foundation (FAPESP), \#312119/2016-0, Conselho Nacional de Desenvolvimento Cient\'ifico  e Tecnol\'ogico (CNPq), and FAEPEX/\\
UNICAMP. CFM was supported by the grants NSF-DMS 2009859 and
NSF-DMS 2239325. RMSR received support under the grants Coordena\c{c}\~ao de Aperfei\c{c}oamento de Pessoal de N\'{\i}vel Superior (CAPES), Brasil, \#001, and Conselho Nacional de Desenvolvimento Cient\'{\i}fico e Tecnol\'ogico (CNPq), Bras\'{\i}lia, Brasil, \#408751/2023-1.


\bibliographystyle{plain}
\bibliography{refs}


\end{document}